\documentclass[english]{amsart}

\usepackage[all,cmtip]{xy}
\usepackage{tikz-cd}

\usepackage{amsmath,amssymb,amscd,amsfonts}

\usepackage{babel}
\usepackage{amstext}
\usepackage{amsmath}
\usepackage{amsfonts}
\usepackage{latexsym}
\usepackage{ifthen}
\usepackage{xypic}
\xyoption{all}
\newcommand{\codim}{{\rm codim}}
\newcommand{\Pic}{{\rm Pic}}

\newcommand\cE{{\mathcal E}}

\newcommand\cF{{\mathcal F}}
\newcommand\cA{{\mathcal A}}

\newcommand\sH{{\mathcal H}}
\newcommand\sI{{\mathcal I}}

\newcommand\bZ{{\mathbb Z}}

\newcommand\bQ{{\mathbb Q}}

\newcommand\Hom{{\rm Hom}}

\newcommand\cO{{\mathcal O}} 
\newcommand\cQ{{\mathcal Q}}

\def\Psef{\mathop{\rm Psef}\nolimits}

\def\Id{\mathop{\rm Id}\nolimits}
\def\rank{\mathop{\rm rank}\nolimits}

\def\Pic{\mathop{\rm Pic}\nolimits}

\def\tor{\mathop{\rm tor}\nolimits}

\def\Im{\mathop{\rm Im}\nolimits}
\def\Mov{\mathop{\rm Mov}\nolimits}

\def\dbar{\overline\partial}
\def\ddbar{i\partial\overline\partial}
\def\cO{{\mathcal O}}

\def\cE{{\mathcal E}}

\def\cF{{\mathcal F}}
\def\cC{{\mathcal C}}
\def\cG{{\mathcal G}}
\def\cX{{\mathcal X}}

\let\ep=\varepsilon

\let\wh=\widehat
\def\bQ{{\mathbb Q}}

\def\bZ{{\mathbb Z}}

\newtheorem{theorem}{Theorem}[section]
\newtheorem{lemma}[theorem]{Lemma}
\newtheorem{claim}[theorem]{Claim}
\newtheorem{proposition}[theorem]{Proposition}
\newtheorem{question}[theorem]{Question}
\newtheorem{conjecture}[theorem]{Conjecture}
\newtheorem{defn}[theorem]{Definition}
\newtheorem{cor}[theorem]{Corollary}
\theoremstyle{remark}

\newtheorem{remark}[theorem]{Remark}

\title{Remarks on Relative Canonical Bundles \\ and Algebraicity Criteria for Foliations \\ in Kähler context}
\author{Junyan CAO, Mihai P\u AUN}

\address{Université Côte d’Azur, CNRS, LJAD, France}
\email{junyan.cao@univ-cotedazur.fr}
\begin{document}

\address{Universit\"at Bayreuth, Lehrstuhl Mathematik VIII, Germany}
\email{mihai.paun@uni-bayreuth.de}

\maketitle

\section{Introduction}

\noindent In this note, motivated by the recent preprint \cite{Ou}, we pursue three main objectives. The first is to make progress towards the fundamental conjecture proposed in \cite{CH} by Cao-Höring. 

\begin{conjecture}\label{conjc-h}\cite{CH} 
	Let $p:X\to Y$ be a holomorphic surjective map between two compact Kähler manifolds. Let $(L, h_L)$ be a line bundle over $X$, endowed with a possibly singular metric $h_L$ such that $i\Theta_{h_L} (L)\geq 0$ on $X$ in the sense of currents. Assume that the following conditions hold:
	\begin{itemize} 
		\item The multiplier ideal sheaf of $h_L$ is trivial.
		\item The restriction of the adjoint bundle $K_X+ L$ to the fiber $X_y$ of $p$ over $y$, denoted as $\displaystyle (K_X+ L)|_{X_y}$, is pseudo-effective for generic $y\in Y$.
	\end{itemize}
	Then the twisted relative canonical bundle $K_{X/Y}+ L$ is pseudo-effective.
\end{conjecture}
\medskip

\noindent In this generality, the main difficulty is that there is no obvious candidate for a natural, positively curved metric on $K_{X/Y}+ L$ (as opposed to the projective case). Nevertheless, we establish the following particular case of Conjecture \ref{conjc-h} (the definition of $D(p)$ is recalled in the next section).

\begin{theorem}\label{thm3}
	In the setting of Conjecture \ref{conjc-h}, further assume that there exists a Kähler metric $\omega$ and a holomorphic 2-form $\sigma$ on $X$ such that $\omega+ \sigma + \overline{\sigma}$ is a rational class and that the restriction of $\sigma$ to the generic fiber of $p$ vanishes. Then $K_{X/Y}+ L- D(p)$ is pseudo-effective.
\end{theorem}

\medskip

\noindent Our second main result provides a simple proof of the algebraicity criterion established in \cite{Ou}.  Ou's algebraicity criterion provides an important extension to the Kähler case of the results in \cite{Bos01, BM16}, and more recently \cite{CP19} (see also \cite{Dru}). The version we prove here is as follows.

\begin{theorem}\label{mainthm}
	Let $X$ be a compact Kähler manifold, and let $C \subset X$ be a compact submanifold. Let $S_0$ be an irreducible, locally closed submanifold containing $C_0$, where $C_0$ is a dense open subset of $C$ satisfying $\codim_C (C\setminus C_0) \geq 2$.  
	Let $M$ be the Zariski closure of $S_0$ in $X$.  
	If $\dim M > \dim S_0$, then the conormal bundle $\mathcal{N}^*_{C_0/S_0}$ is pseudo-effective (abbreviated as "psef") in the sense of Definition \ref{anapsf}.
	
	In particular, let $\cF \subset T_X$ be a holomorphic foliation on $X$. If $\cF^\star$ is not psef, then $\cF$ is induced by a meromorphic map. 
\end{theorem}

\medskip

\noindent Finally, our third objective is to slightly extend W. Ou’s uniruledness criterion \cite[Thm 1.1]{Ou}.
This generalization is made possible by Theorems \ref{thm3} and \ref{mainthm}. Our proof largely follows the approach in \cite{CP19}. In particular, we establish the following results.

\begin{theorem}\label{quot}
	Let $X$ be a compact Kähler manifold whose canonical bundle is pseudo-effective. Consider a coherent, torsion-free sheaf $\cQ$ such that we have
	\[
	\otimes^m T_X^\star \to \cQ \to 0
	\]
	for some $m \geq 1$.  
	Then $\det \cQ$ is a pseudo-effective line bundle. 
\end{theorem}

\begin{cor}\label{RCC} 
	Let $X$ be a compact Kähler manifold. Let $\cF \subset T_X$ be a holomorphic foliation. If $\mu_{\alpha, \min}(\cF) > 0$ for some $\alpha \in \Mov(X)$, 
	then $\cF$ is  algebraic and its leaves are rationally connected.
\end{cor}

\medskip

\noindent Next, we provide a few comments on the proofs of these results.  

In the context of Theorem \ref{thm3}, the map $p: X\to Y$ is locally projective,  
at least generically over the base $Y$. The usual techniques—using Bergman kernels and their $L^{2/k}$ versions (cf. \cite{BP})—allow us to construct metrics on $\displaystyle (K_{X/Y}+ L)|_{p^{-1}(U)}$ for any coordinate subset $U\subset Y$ centered at a generic point $y\in Y\setminus Z$, where  
$Z\subset Y$ is an analytic subset. However, a priori, it is unclear how to construct these metrics in such a way that they coincide on overlapping subsets; this represents the main challenge here. Indeed, once this issue is resolved,  
the resulting metric extends across $Z$ (by the familiar estimates of their weights). The key ideas of the proof are contained in \cite{HP24} and \cite{MP2}.

\medskip

The main part of the proof of Theorem \ref{mainthm} closely follows the approach in \cite[Thm 1.2]{Ou}. The main difference is that \cite{Ou} requires certain algebraic constructions to carefully desingularize the singular locus of 
$M$ and to control the loss of the pseudo-effectiveness of the conormal bundle in a highly efficient manner. In our proof, we adopt an analytic approach by comparing the Lelong numbers of a certain current and its pullback (cf. Proposition \ref{controllelong}). This makes the proof purely analytic.
Note that the analytic definition of pseudo-effectiveness for sheaves Definition \ref{anapsf} is a priori slightly weaker than the algebraic version given in \cite{Ou, HP19}. However, thanks to Proposition \ref{slope}, Theorem \ref{mainthm} still implies the main geometric applications, such as \cite[Thm 1.1, Thm 1.4]{Ou}, Theorem \ref{quot}, and Theorem \ref{main}.

Let us explain briefly the idea of the proof of Theorem \ref{mainthm}, i.e., \cite[Thm 1.2]{Ou}. By using Demailly's mass concentration method \cite{Dem93} and Collins-Tosatti's extension theorem \cite{CT}, we construct an $\omega_X$-psh function $\varphi$ on $X$ such that the Lelong number $\nu(\varphi |_{S_0}, y_0)$ at  a fixed point $y_0\in C_0$ can be arbitrarily large cf. Proposition \ref{lelong}. To the best of our knowledge, in algebraic context this was first introduced by Bost in \cite{Bos04}. In case the ambient manifold is only assumed to be Kähler, the Monge-Ampère equation is usually used as substitute for the 
absence of an ample line bundle, see e.g.  \cite{Dem93, DP, MP1}. It is also the approach followed in 
\cite{Ou}. The main point is to show that as soon as $\nu(\varphi |_{S_0}, y_0)$ is large, so is the generic Lelong number $\nu (\varphi |_{S_0}, C_0)$ of $\varphi$ along $C_0$ (hence, no isolated singularities!), see Proposition \ref{arbitrarylelong}. At least if $M$ is non-singular, this claim is clear: very vaguely, it is a consequence of the fact that for \emph{compact} manifolds, the Lelong numbers of a closed positive current are bounded in terms of its cohomology class.

Let us also mention that given a foliation $\cF$ by curves on a Kähler manifold, the most complete algebraicity result is due to M. Brunella, cf.\cite{Brun1, Brun2}.  He shows e.g. that if the generic leaf of $\cF$ is uniformized by the unit disk, then the leaf-wise Poincaré metric has a psh variation. This represents a wide generalisation of all known statements in this framework -but again, it only applies to foliations by curves. 
%In the general case, when $M$ is not smooth, we need additional arguments.
%Let $\pi: \widehat{X} \to X$ be the blow up of $C$ in $X$. Let $\widehat{S}_0$ be the strict transform of $S_0$. In Proposition \ref{addlemma}, we construct $\omega_{\widehat{X}}$-psh functions $\varphi$ on  $\widehat{X}$ and show that the Lelong number of $\varphi |_{\widehat{S}_0}$ along $\widehat{S}_0\cap \pi^{-1} (C_0)$ can be arbitrarily large.  Once this is established, we complete the proof of the main theorem using standard arguments.

\medskip

As we have already mentioned, the proof of Theorem \ref{quot} and Corollary \ref{RCC} follow from the arguments in \cite{CP19}, combined with additional inputs from \cite{C}.
\medskip

\noindent \textbf{Acknowledgements} JC would like to thank Wenhao Ou for sending the preliminary versions of his very nice paper \cite{Ou}, as well as for many discussions related to his paper and  this note. He thanks the Institut Universitaire de France and the A.N.R JCJC project Karmapolis (ANR-21-CE40-0010) for providing excellent working condition. MP gratefully acknowledge support from the DFG. Special thanks are due to F. Campana: virtually all the results in Section 5 (together with the correct parts of the proofs...) were suggested by him.

\section{Some Preliminary Results}

We collect next a few properties of quasi-plurisubharmonic (quasi-psh) functions and closed positive currents which will be needed later.

\begin{proposition}\label{listproperty}
	Let $X$ be a complex manifold, and let $X_0\subset X$ be an open set such that $\codim_X (X\setminus X_0) \geq 2$. Let $Y \subset X$ be a submanifold. Then we have:
	
	\begin{itemize}
		\item El Mir extension theorem \cite[Chapter III, Thm 2.3]{Dem} Let $T$ be a $d$-closed positive $(1,1)$-current on $X_0$. Then $T$ can be extended to a positive current on the entire space $X$. 
		
		\item Support theorem \cite[Chapter III, Thm 2.10]{Dem} Let $T_1$ and $T_2$ be two
		$d$-closed positive $(1,1)$-currents on $X$ such that $T_1 =T_2$ on $X_0$. Then $T_1 =T_2$ on $X$.
		
		\item \cite{Hor} Let $0\in \mathbb{C}^n$ be the origin. Let $\varphi$ be a psh function in a neighborhood of $0$. Then the Lelong number $\nu (\varphi , 0)$ can be computed using the following formula:
		\begin{equation}\label{restrictionprop}
			\nu (\varphi , 0) = \min \{\nu (\varphi |_{L_\xi}, 0) \mid \xi \in \mathbb{C}^n \setminus \{0\}\},
		\end{equation}
		where $L_\xi$ is the complex line $\mathbb{C} \cdot \xi$. In fact, $\nu (\varphi , 0) =\nu (\varphi |_{L_\xi}, 0)$ for almost every $\xi \in \mathbb{C}^n \setminus \{0\}$.
		
		As a consequence, for psh functions with analytic singularities, namely $\varphi= c\log \sum_{i=1}^m |f_i|^2 + O (1)$,
		where $f_i$ are holomorphic functions, $\frac{1}{2c} \nu (\varphi , 0)$
		equal to the minimal vanishing order of $(f_1, \dots, f_m)$.
		
		\item Let $\varphi$ be a quasi-psh function on $X$ such that $\varphi$ is not identically $-\infty$ on $Y$. 
		Let $\nu (\varphi , Y)$ be the Lelong number of $\varphi$ along $Y$, defined as the Lelong number of $\varphi $ at a generic point of $Y$. By the semicontinuity of Lelong numbers, we have  
		$\nu (\varphi, Y) = \min_{y\in Y} \nu (\varphi, y)$.
		
		\item Let $\pi: \widehat{X} \to X$ be the blow-up along $Y$. Let  $x\in X$. If $\varphi$ is a quasi-psh function on $X$ with analytic singularities, by \eqref{restrictionprop} we have 
		$$\nu (\varphi, x) = \nu (\pi^*\varphi, \pi^{-1} (x)) , \quad \text{and} \quad \nu (\varphi, Y) = \nu (\pi^*\varphi, \pi^{-1} (Y)) .$$ 
		
		\item \cite[Prop 5.10]{Dem12} Let $\varphi$ be a psh function on $X$. Let $\mathcal{I} (\varphi)$ be the multiplier ideal sheaf generated locally by holomorphic functions $f$ such that
		$$\int |f|^2 e^{-2\varphi} dV_{\omega_X} <+\infty.$$
		If $\varphi= c\log \sum_{i=1}^n  |z_i|^{2\alpha_i}$ for real numbers $\alpha_i >0$, then $\mathcal{I} (\varphi)_0$ is generated by the monomials $\prod z_i^{\beta_i}$ such that
		$$\sum_{i=1}^n \frac{\beta_i +1}{\alpha_i } >2c .$$
	\end{itemize}
\end{proposition}

We will now discuss the Lelong number of quasi-psh functions after a bimeromorphic transformation.

\begin{lemma}\label{controllelong}
	Let $X$ be a complex manifold, and let $Y\subset X$ be a submanifold. Let $\pi: \widehat{X}\to X$ be the blow-up of $Y$, and let $E \subset \widehat{X}$ be the exceptional divisor.
	Let $\widehat{T} \geq 0$ be a positive $(1,1)$-current on $\widehat{X}$ with analytic singularities.
	Let $\widehat{x} \in E$ and set $x := \pi (\widehat{x})$. Define $T:=\pi_* (\widehat{T})$.
	Then we have  
	$$C\cdot  \nu (T, x) \geq \nu (\widehat{T},  \widehat{x}) - \nu (\widehat{T}, E)  .$$
	for some uniform constant $C$ depending only on $X$ and $Y$.
\end{lemma}

\begin{proof}
	Set $T_1 := \pi^* T$. Then $T_1$ is a positive current on $ \widehat{X}$ and coincides with $\widehat{T}$ outside $E$. 
	Let $a:= \nu (\widehat{T}, E)$ and $b:=\nu (T_1, E)$.
	Then $(T_1 - b [E]) - (\widehat{T} - a[E])$ is a current supported on a subvariety of codimension at least $2$. By Proposition \ref{listproperty}, we have  
	$$T_1 -b[E]\equiv \widehat{T} - a[E] \qquad\text{ on } \widehat{X} .$$
	Using \cite[Thm 2]{Fav}, we know that $\nu (T_1, \widehat{y}) \leq C \nu (T, y)$ for some constant $C$. 
	Therefore, we obtain  
	$$\nu (\widehat{T},  \widehat{x}) -a = \nu (\widehat{T} - a[E], \widehat{x}) =\nu (T_1 - b[E], \widehat{x}) \leq\nu (T_1 , \widehat{x} )\leq C \nu (T, x) .$$
\end{proof}

\begin{lemma}\label{controldirectimage}
	Let $X$ be a complex Kähler manifold, and let $V$ be a holomorphic vector bundle of rank $r$ on $X$. Let $\pi: \mathbb{P} (V) \to X$ be the natural projection. Fix a point $x\in X$. 
	Let $\omega$ be a Kähler metric on $\mathbb{P} (V)$. 
	
	Let $T \geq 0$ be a $d$-closed positive $(1,1)$-current on $\mathbb{P} (V)$. Consider the $(1,1)$-positive current $T':= \pi_* (T\wedge \omega^{r-1})$ on $X$. Then we have
	$$\nu (T', x) \geq  (\int_{\pi^{-1} (x)} \omega^{r-1}) \cdot \nu (T, \pi^{-1} (x)) . $$
\end{lemma}

\begin{proof}
	We first assume that $\dim X=1$. Set $a:= \nu (T, \pi^{-1} (x))$. Then $T_a:= T - a [\pi^{-1} (x)]$ is a $(1,1)$-positive current. We have  
	$$T'=\pi_* (T_a \wedge \omega^{r-1})  + a \left(\int_{\pi^{-1} (x)} \omega^{r-1} \right) [x] \geq a\left(\int_{\pi^{-1} (x)} \omega^{r-1} \right) [x] . $$
	Thus, the lemma is proved.
	
	If $\dim X\geq 1$, we take a generic one-dimensional unit disc $\Delta \subset X$ passing through $x$. By \eqref{restrictionprop}, we have $\nu (T' , x)= \nu (T' |_{\Delta}, x)$. Thus, the proposition is reduced to the one-dimensional case.
\end{proof}

\begin{defn}\label{anapsf}
	Let $X$ be a compact complex manifold, and let $\mathcal{F}$ be a reflexive sheaf on $X$. Let $X_0 \subset X$ be the locally free locus of $\mathcal{F}$.  Consider the projection $\pi: \mathbb{P} (\mathcal{F} |_{X_0}) \to X_0$, and let $Y$ be a desingularization of $\mathbb{P} (\mathcal{F})$ with a hermitian metric $\omega_Y$. Then $\mathbb{P} (\mathcal{F} |_{X_0})$ can be embedded in $Y$.
	
	We say that $\mathcal{F} |_{X_0}$ is psef if, for every $\epsilon > 0$, there exists a possibly singular metric $h_\epsilon$ on $\mathcal{O}_{\mathbb{P} (\mathcal{F} |_{X_0})} (1)$ 
	such that  $i\Theta_{h_\epsilon} (\mathcal{O}_{\mathbb{P} (\mathcal{F} |_{X_0})} (1)) \geq -\epsilon\omega_Y $ on $\mathbb{P} (\mathcal{F} |_{X_0})$.
\end{defn}

Since $\mathbb{P} (\mathcal{F} |_{X_0})$ is not compact (its complement may have codimension $1$), it is not clear whether the above definition implies the existence of a metric $h$ such that  
$i\Theta_{h} (\mathcal{O}_{\mathbb{P} (\mathcal{F} |_{X_0})} (1)) \geq 0 $ on $\mathbb{P} (\mathcal{F} |_{X_0})$.
Thus, Definition \ref{anapsf} is à priori weaker than \cite[Definition 4.1]{Ou} and \cite{HP19}. 
However, following a similar argument as in \cite[Section 4]{Ou}, we can still prove that the maximal slope $\mu_{\alpha, \max} (\mathcal{F}) \geq 0$ for any mobile class $\alpha \in \Mov (X)$ cf. Proposition \ref{slope} in appendix.

\section{Positivity of relative canonical bundles in K\"ahler context}
\medskip
\noindent In this section, we consider the following setup:
\begin{enumerate}
	\item[(i)] $X$ and $Y$ are compact K\"ahler manifolds, and $p: X\to Y$ is a holomorphic, surjective map.
	\smallskip
	
	\item[(ii)] There exist a K\"ahler metric $\omega$ and a holomorphic (2,0)-form $\sigma$ on $X$ such that  
	\[\omega+ \sigma+ \overline{\sigma}\]
	represents the Chern class of a topological line bundle. 
	\smallskip
	
	\item[(iii)] For each generic point $y\in Y$, we have $\sigma|_{X_y}= 0$, where $X_y= p^{-1}(y)$ is the fiber of $p$ above $y$. In particular, the morphism $p$ is locally projective over some Zariski open subset of $Y$.
\end{enumerate}
\medskip

\noindent The divisor $D(p)$ associated with the map $p$ is defined as follows:
\[
D(p)= \sum (m_k-1)W_k,
\]
where $W_k$ are irreducible hypersurfaces of $X$ such that $p(W_k) = Z_k$ is a hypersurface of $Y$, and $m_k$ is the multiplicity of $p^{-1}(Z_k)$ along $W_k$. 

\noindent We now state the main result established in this section.

\begin{theorem}\label{thm1}
	Assume that hypotheses (i)--(iii) above are satisfied. 	
	Moreover, let $(L, h_L)$ be a line bundle such that the following conditions hold:
	\begin{itemize}
		\item The curvature current $\sqrt{-1}\Theta(L, h_L)\geq 0$ is positive on the total space $X$, and the multiplier ideal sheaf of $h_L$ is trivial.
		
		\item 
		The adjoint bundle $\displaystyle K_{X_y}+ L|_{X_y}$ is pseudo-effective for each $y\in Y\setminus W$, where $W\subsetneq Y$ is a proper analytic subset of $Y$.
	\end{itemize}
	Then the twisted relative canonical bundle \[K_{X/Y}+ L- D(p)\] is pseudo-effective as well.
\end{theorem}

\begin{remark}
	{\rm 
		Of course, we expect this result to hold in general, without assuming the existence of the decisive form $\sigma$ as in (ii)--(iii) above.
		A version of this was conjectured by Cao-H\"oring in \cite{CH}. We refer to \cite{MP1} and \cite{HG} for some results in this direction. }
\end{remark}

\noindent The main ideas for the proof of this statement are contained in the articles 
\cite{HP24} and \cite{MP2}. However, since Theorem \ref{thm1} does not appear to be stated explicitly anywhere, we will provide an almost complete treatment in what follows. 

\noindent In the current context, we can use the fact that $p$ is locally projective to construct a positively curved metric on $K_{X/Y}+ L- D(p)$ over a coordinate subset of the base $Y$. The main new difficulty we address here is the dependence on the chosen subset: the metric we construct relies on the local projectivity of $p$, but in the end, we show that it is independent of this choice using the following lemma.

\begin{lemma}\label{minmetrc}
	Let $Z$ be a projective manifold, and let $(F, h_F)$ be a holomorphic line bundle with a possibly singular metric $h_F$ such that the adjoint bundle 
	$K_Z+ F$ is pseudo-effective, the curvature $i\Theta_{h_F} (F) \geq 0$ in the sense of currents, and the multiplier ideal sheaf  $\mathcal{I} (h_F)= \mathcal O_Z$.
	
	\noindent Let $A$ be an ample line bundle on $Z$ with a smooth metric $h_A$.
	For each topologically trivial line bundle $\rho \in \Pic^0 (Z)$, let $h_\rho$ be a metric on $\rho$ such that $i\Theta_{h_\rho} (\rho) \equiv 0$.
	
	\noindent Fix an integer $m_0\in \mathbb{N}^\star$. For each sufficiently divisible integer $k$, we consider the space of sections:
	\[
	V_{k, \rho}:= H^0\big(Z, k(K_Z+ F+ {1}/{m_0}A)+ \rho\big).
	\]
	
	\noindent Let $e$ be a local basis of the $\mathbb{Q}$-bundle $K_Z+ F+ 1/m_0A$. We define two extremal metrics on $K_Z+ F+ 1/m_0A$:
	\[
	|e|^2 _{h_{\rm min}} (x) := \inf_{k, \rho, s} \Bigg\{ \frac{|e|^2 (x)}{|s|_\rho^{\frac{2}{k}} (x)} : s\in V_{k, \rho} \text{ such that } \int_Z|s|^{\frac{2}{k}}_{\rho, h_F, h_A} = 1\Bigg\},
	\]
	and
	\[
	|e|^2 _{h_{\rm min, \rho}} (x) := \inf_{k, s} \Bigg\{ \frac{|e|^2 (x)}{|s|_\rho^{\frac{2}{k}} (x)} : s\in V_{k, \rho} \text{ such that } \int_Z|s|^{\frac{2}{k}}_{\rho, h_F, h_A} = 1\Bigg\},
	\]
	where the subscript $|\cdot |_{\rho, h_F, h_A}$ indicates that we use the metric $h_\rho$ (resp. $h_F, h_A$) on the bundle $\rho$ (resp. $F, A$). 
	
	\noindent Then we have
	\[
	\displaystyle h_{\rm min}= h_{\rm min, \rho} \qquad\text{ for each }\rho \in \Pic^0 (Z).
	\]
\end{lemma}

\begin{proof}
The argument is almost identical to the one used in the last section of \cite{MP2}. The only difference is the manner in which the sections are normalized: in \emph{loc. cit.}, we use the supremum instead of the $L^{2/k}$-norm. 
Note that, by uniform global generation (cf. \cite{Siu}), we have $V_{k,\rho} \neq 0$ when $k$ is sufficiently divisible. 
\smallskip

\noindent  
Fix $\rho_0 \in \Pic^0 (Z)$. We obviously have $\displaystyle h_{\rm min} \leq h_{\rm min, \rho_0}$. 
Now, we prove that $\displaystyle h_{\rm min} \geq h_{\rm min, \rho_0}$. For each sufficiently divisible $k_0 \in \bZ_+$ and for each $\rho \in \Pic^0 (Z)$, we consider a section
\begin{equation}\label{eq1}
	s \in H^0\big(Z, k_0(K_Z+ F+ 1/m_0A)+ \rho\big).
\end{equation}
Since $\mathcal{I} (h_F) = \mathcal{O}_Z$, we know that $\int_Z |s|^{\frac{2}{k_0}} _{\rho, h_F, h_A}  < +\infty$. After normalization, we can assume that $\int_Z |s|^{\frac{2}{k_0}} _{\rho, h_F, h_A} =1$. 

Given any sufficiently large $k \gg 0$ such that $k = k_0 m_0 d$ for some $d \in \bZ$, we write
\[
k(K_Z+ F+ 1/m_0A)+ \rho_0= K_Z+ F+ {(k-1)}\big(K_Z+ F+ 1/m_0A\big)+ \frac{1}{m_0}A+ \rho_0
\]
and use the appropriate power of the section $s$ in \eqref{eq1} (together with the flat metric on $\rho$) to define a singular metric on \[(k-1)\big(K_Z+ F+ 1/m_0A\big)\]
whose curvature is positive.

	\medskip
	
\noindent By Shokurov's trick (cf. \cite{Sho}), we infer the existence of a section 
\[
u\in H^0\big(Z, k(K_Z+ F+ 1/m_0A)+ \rho_0\big)
\]
such that the equality
\begin{equation}\label{eq2}
	\int_Z\frac{|u|_{\rho_0 , h_F, h_A}^2}{|s|_{\rho, h_F, h_A}^{2\frac{k-1}{k_0}}}= 1
\end{equation}
holds. More precisely, this can be seen as follows. Let $\sI$ be the multiplier ideal sheaf corresponding to the weight
$\frac{k-1}{k_0}\log |s|^2$. We clearly have
\[
\chi\big(Z, \big(k(K_Z+ F+ 1/m_0A)+ \rho_0\big)\otimes \sI\big)= \chi\big(Z, \big(k(K_Z+ F+ 1/m_0A)+ \frac{k}{k_0}\rho\big)\otimes \sI\big).
\]
By the Nadel vanishing theorem, the higher cohomology groups vanish.  
Thus, we obtain
\begin{equation}\label{eq20}
	H^0\big(Z, \big(k(K_Z+ F+ 1/m_0A)+ \rho_0\big)\otimes \sI\big)= H^0\big(Z, \big(k(K_Z+ F+ 1/m_0A)+ \frac{k}{k_0}\rho\big)\otimes \sI\big).
\end{equation}
Since the section $s^{\frac{k}{k_0}}$ belongs to the right-hand side (RHS) of \eqref{eq20}, it follows that, in particular, the left-hand side (LHS) is nonzero. This shows the existence of the section $u$, as claimed in \eqref{eq2}.

In particular, $u$ will be divisible by a large power of $s$. More precisely, we can write  
\begin{equation}\label{eq3}
	u= s^{n_k} v_k
\end{equation}
where 
$n_k := m_0 d - 1 = \lfloor \frac{k-1}{k_0} \rfloor$, and $v_k$ is a section of the bundle  
\[
k_0(K_Z+ F+ 1/m_0A) + \rho_0 - n_k\rho,
\]
whose Chern class is uniformly bounded with respect to $k$, since $c_1(\rho) = 0$.  
Moreover, Hölder's inequality, combined with \eqref{eq2}, shows that we have  
\begin{equation}\label{eq4}
	\log \int_Z |u|^{\frac{2}{k}}_{\rho_0} e^{-\varphi_F - \frac{1}{m_0} \varphi_A} 
	\leq \frac{k-1}{k} \log \int_Z |s|^{\frac{2}{k_0}}_\rho e^{-\varphi_F - \frac{1}{m_0} \varphi_A} = 0.
\end{equation}

	\medskip
	
	Now we use  the arguments in  \cite{MP2}. 
	By  \eqref{eq2}, we have 
	\begin{equation}\label{eq30}
		\int_Z\frac{|v_k|^2_{\rho}}{|s|^{2\frac{k_0-1}{k_0}}}= 1
	\end{equation} 
Let $h_0$ be an arbitrary, smooth metric on the bundle $k_0(K_Z+ F+ 1/m_0A)$. We consider the function
\begin{equation}\label{eq31}
	\phi_k:= \log |v_k|_{h_0, \rho}^2
\end{equation} 
and then there exists a constant $C> 0$ such that we have
\begin{equation}\label{eq32}
	dd^c\phi_k\geq -C\omega,
\end{equation}
for all $k$. Then the boundedness of $k_0(K_Z+ F+ 1/m_0A)+ \rho_0- n_k\rho$ implies that
\begin{equation}\label{eq33}
	-\log {C}\leq \max_Z \phi_k\leq \log C
\end{equation}
for some larger constant $C$. The RHS of \eqref{eq33} is clear -as consequence of \eqref{eq30}. As for the second one, we refer to \cite{MP2}, last part of Section 2.1.

	\smallskip
	
	\noindent This in turn implies that any limit extracted from $(\phi_k)_k$ will be a quasi-psh function, 
	non-identically $-\infty$. We note by $f_s$ the local holomorphic function corresponding to the section $s$ with respect to a trivialisation of $K_Z, F,  A$ and $\rho$. 
	Then we have 
	\begin{equation}\label{eq34}
		\frac{1}{k}(\log |f_u|^2- \varphi_{\rho_0})= \frac{n_k}{k}(\log|f_s|^2- \varphi_\rho)+ \frac{1}{k}\big(\phi_k+ \cO(1)\big)
	\end{equation}
	as consequence of \eqref{eq3}, where the quantity $\cO(1)$ in \eqref{eq34} corresponds to the weights of the metric 
	$h_0$, so it is independent of $k$.
	\smallskip
	
	\noindent By taking the limit as $k\to\infty$, we obtain
	\begin{equation}\label{eq5}
		\varphi_{min, \rho_0}(x)\geq \frac{1}{k_0}(\log|f_s|^2- \varphi_\rho)  ,
	\end{equation} 
	where $\varphi_{min, \rho_0}$ is the weight function of $	h_{min, \rho_0}$.
This implies that the inequality
	\begin{equation}\label{eq6}
	h_{min, \rho_0}\leq h_{min}
	\end{equation} 
	is also verified. As already mentioned, the opposite inequality is obvious, so all in all, the lemma is established.
\end{proof}

We are now ready for the proof of Theorem \ref{thm1}.
\begin{proof}[Proof of Theorem \ref{thm1}] 
	
We have divided our arguments in a few steps. The main claim is emphasised at the beginning of each step.
\medskip

\noindent $\bullet$ {\bf Step one.} \emph{Let $\cA$ be the line bundle on $X$ given by {\rm (ii)}. For each $y\in Y$ generic, the restriction $\displaystyle \cA|_{X_y}$ admits a holomorphic structure, and a Hermitian metric whose curvature is equal to the $\displaystyle \omega|_{X_y}$. Moreover, given any $y_0\in Y$ generic there exists an open coordinate subset $y_0\in U\subset Y$ and an ample line bundle $\sH_U$ on $p^{-1}(U)$ together with a positively curved metric $h_U$ whose local weights are given by potentials of $\displaystyle \omega|_{p^{-1}(U)}$. }

\noindent The claim above follows by the familiar arguments proving that for any representative of the Chern class of a holomorphic line bundle there exists a Hermitian metric whose curvature is precisely the said representative.
%\footnote{MP: des explications plus détaillées ici?}
\medskip

\medskip

\noindent $\bullet$ {\bf Step two.} \emph{Psh variation of metrics with minimal singularities} Let $p: \cX\to U$ be the co-restriction of $p$ to the coordinate set $U\subset Y$, cf. Step 1. We recall the following statement.
\begin{theorem}\cite{BP}\label{thm2} 
Assume that the hypothesis in Theorem \ref{thm1} are satisfied.  
Then for all integers $m_0\geq 1$ the $\bQ$-bundle $K_{\cX/U}+ L+  1/m_0\sH_U$ is pseudo-effective. More precisely, the metric whose restriction to the generic fiber $X_y$ is given by $h_{\rm min}$ defined  Lemma \ref{minmetrc}, applied to the following data: \[Z:= X_y,\quad F:= L|_{X_y}, \quad A:= \sH_U|_{X_y}\]
induces a positively curved metric on $K_{\cX/U}+ L+  1/m_0\sH_U$.
\end{theorem}

\noindent In other words, the weight of the metric $h_{\rm min}$ constructed in Theorem \ref{thm2}  is psh.  Now let $U_1$ and $U_2$ be two coordinate open subsets such that $U_1\cap U_2\neq \emptyset$. It is clear that in general we will have
\[\sH_{U_1}|_{X_y}\neq \sH_{U_2}|_{X_y}\]
but the Chern class of these restrictions coincide with the class defined by some multiple of $\displaystyle \{\omega|_{X_y}\}$. Therefore the metrics we constructed on $K_{\cX/U_1}+ L+  1/m_0\sH_{U_1}$ and 
$K_{\cX/U_2}+ L+  1/m_0\sH_{U_2}$ are the same, and thus we have showed that the restriction of the class
\[c_1(K_{X/Y}+ L)+ 1/m_0\{\omega\}\]
to $X\setminus p^{-1}(W)$ contains a \emph{natural} closed positive current. 
\medskip

\noindent $\bullet$ {\bf Step three.}  We would like to clarify the concerns with respect to the divisor $D(p)$. It is based on the following claim.

\begin{claim} Let $h_{\cX/U}^{(k)}$ be the fiberwise-$L^{2/k}$ metric induced by the space of sections $V_{k, \rho}$ at the generic point of $U$. 
Then the corresponding (renormalised) curvature current in the class $\{K_{\cX/U}+ L+  1/m_0\sH_U\}$ is greater than
\[\sum (m_i- 1)[W_i\cap\cX].\]
\end{claim}

\noindent  To prove the claim, let $\Omega\subset \cX$ be some open set such that the bundles $L$ and $\sH_U$ are equally trivial when restricted to $\Omega$. We fix a coordinate system $z= (z_1,\dots, z_{n+ s})$ on $\Omega$ and $t=(t_1,\dots, t_s)$ on $U$. This provides in particular a trivialisation of the relative canonical bundle $K_{\cX/U}$ restricted to $\Omega$. 
Then we have
\[e^{\varphi_{\cX/U}^{(k)}(x)}= \sup_{\Vert u\Vert_{k, y}=1}|F_u(x)|^2\]
where the notations are as follows:
\begin{itemize}

\item $x\in \Omega$ and $y= p(x)$ is a generic point of $U$. Here "generic" has a very precise meaning: $y$ should be a regular point of $p$ and the sections of $V_{k, \rho}$ corresponding to the fiber $X_y$ should extend to the neighbouring fibers. 

\item The $L^{2/k}$- norm is defined by $\displaystyle \Vert u\Vert_{k, y}^{\frac{2}{k}}:= \int_{X_y}|u|^{\frac{2}{k}}e^{-\varphi_L- \frac{1}{m_0}\varphi_\sH}$ (which is why this integral appears in Step two).

\item Finally, the holomorphic function $F_u$ is defined by the equality $\displaystyle u\wedge p^{\star}dt^{\otimes k}= F_u dz^{\otimes k}$ between two twisted pluricanonical forms. Here $dz:= dz_1\wedge\dots\wedge dz_{n+s}$, and a similar convention for $dt$. 
\end{itemize}

\noindent Now, since $dd^c \varphi_{\cX/U}^{(k)}\geq 0$, it would be sufficient to analyse the singularity of this psh function near the generic point of $W_i\cap \Omega$. We therefore can assume that the coordinates have been chose so that the map $p$ is given by the expression.
\[p(z_1,\dots, z_{n+s})= (z_1^{m_i}, z_2,\dots, z_s)\]
It then follows that the quotient $\displaystyle \frac{F_u}{z_1^{m_i-1}}$ is in $L^2$, and moreover its norm on $X_y\cap \Omega$ is uniformly bounded independently as $y$ tends to $p(W_i)$. It follows that 
\[dd^c \varphi_{\cX/U}^{(k)}\geq (m_i-1)[W_i\cap\cX]\]
as claimed (see \cite[Thm 2.3]{CP17} for a more detailed explanation).
\medskip

\noindent $\bullet$ {\bf Step four.} \emph{Extension.} As consequence of the previous steps, we have constructed a closed positive current  
\[\Xi_0= \Theta+ \frac{1}{m_0}\omega+ dd^c\psi\] 
on $X\setminus p^{-1}(W)$, the inverse image of the complement of an analytic subset $W$ of $Y$. 
Here $\Theta$ is a smooth representative of the first Chern class of $K_{X/Y}+ L$. We recall next the "usual arguments" showing that $\Xi_0$ extends across $p^{-1}(W)$. 

Consider the potential $\varphi_{\cX/U}^{(k)}$ of the current $\Xi_0$, as recalled in the previous step. If we are able to show that this function is bounded from above, it follows that $\Xi_0$ extends across the inverse image of $W$, see \cite{Dem}. 

Assume that $u$ is a section of $\displaystyle k(K_{X_y}+ L+ 1/m_0\mathcal H_U|_{X_y})$, as fix a point $x\in X_y$ is the fiber $X_y$ above a generic point $y\in U$. Let $x\in \Omega\subset X$ be a coordinate subset as in Step three. By the local $L^{2/k}$-version of Ohsawa-Takegoshi extension theorem \cite{BP}, there exists a local holomorphic function $F_\Omega$ such that 
\[F_\Omega|_{\Omega\cap X_y}= \frac{u\wedge p^{\star}dt^{\otimes k}}{dz^{\otimes k}}, \qquad \int_{\Omega}|F_{\Omega}|^{2/k}e^{-\varphi_L- \frac{1}{m_0}\varphi_\sH}d\lambda(z)\leq C_0\]
where $C_0$ is a numerical constant (in particular, independent of $k$).

Since $\Omega$ is some small Stein open set, $\omega |_{\Omega}$ is $\ddbar$-exact. Then the local weight $\varphi_{\sH}$ of $\sH_U$ on $\Omega'$ can be uniformly bounded for any open subset $\Omega' \Subset \Omega$ (no matter how close we are from the inverse image of $W$). We therefore infer that there exists a constant $C_1$ so that 
\[\sup_{\Omega'}\frac{1}{k}\log|F_\Omega|^2\leq C_1\]
holds. Finally, the restriction of $F_\Omega$ to the fiber $X_y$ is equal to 
$F_u$, and therefore we have shown that 
\[\sup_{\Omega'\setminus p^{-1}(W)}\varphi_{\cX/U}^{(k)}< \infty.\]
Then $\Xi_0$ can be extended as a positive current on $X$. 
The proof of Theorem \ref{thm1} is finished by taking the limit $m_0\to\infty$.
\end{proof}
\medskip

\begin{remark}
Most likely, the type of arguments in this section can be used in order to prove Conjecture \ref{conjc-h} for locally projective families $p$. There are however quite a few difficulties to overcome.
\end{remark}

%%%%%%%%%%%%%%%%%%%%%%%%%%%%%%%%
%%%%%%%%%%%%%%%%%%%%%%%%%%%%%%%%%

		\section{Proof of Theorem \ref{mainthm}}
		Let $X$ be a compact Kähler manifold, and let $S_0$ be a locally closed submanifold of $X$. Fix a point $x \in S_0$. When $X$ is projective, fix an ample line bundle $A$ on $X$. Bost \cite{Bos04} proved a characterization of the algebraicity of $S_0$ by studying the subspaces $E^i_m \subset H^0(X, mA)$ consisting of sections whose restrictions to $S_0$ vanish to order at least $i$ at $x$. In \cite{Ou}, Ou generalized Bost's nice idea to the setting of compact Kähler manifolds by studying the restricted Lelong number $\nu (\varphi |_{S_0}, x)$, where $\varphi$ is an $\omega_X$-plurisubharmonic function on $X$ for some fixed Kähler metric $\omega_X$.
		For the reader's convenience, we first recall his elegant proof, which follows precisely the same argument as the first part of \cite[Thm 1.2]{Ou}.
		
		\begin{proposition}\label{lelong}
			Let $X$ be a compact Kähler manifold, and let $\omega_X$ be a Kähler metric on $X$. Let $S_0$ be a locally closed submanifold of $X$ such that $\dim S_0 < \dim X$ and the Zariski closure of $S_0$ is $X$. Fix a point $x \in S_0$.  
			
			Then, for any $\lambda > 0$, we can find a $\omega_X$-psh function $\varphi$ on $X$ such that $\varphi$ has analytic singularities and satisfies $\nu (\varphi |_{S_0}, x) \geq \lambda$.
		\end{proposition}
		
\begin{proof}

	Let $U$ be a neighborhood of $x$ in $X$, and let $V$ be a neighborhood of $x$ in $S_0$.  
	Since $\dim S_0 < \dim X$, we can find a local holomorphic coordinate system $(z_1, \dots, z_n)$ on $U$ such that $V$ is contained in $\{z_1=0\}$.
	
	Let $m \in \mathbb{N}$. We consider the function  
	$$f = \frac{1}{m^{\frac{n-1}{n}}} \log (|z_1|^2 + |z_2|^{2m} + \cdots + |z_n|^{2m}).$$  
	Since $V$ is contained in $\{z_1=0\}$, we have $\nu (f |_{S_0}, x) = 2 m^{\frac{1}{n}}$.  
	By construction, the mass $(\ddbar f)^n = A \cdot \{x\}$ for some uniform constant $A$ independent of $m$.  
	By applying Demailly's mass concentration theorem \cite[Cor 6.8]{Dem93}, we can find a $2\omega_X$-psh function $\psi$ on $X$ such that  
	$$\psi \leq C \cdot f + O(1)$$  
	in a neighborhood of $x$, where $C$ is a uniform constant independent of $m$.  
	\medskip
	
	Next, we apply Demailly's regularization theorem to $\psi$. We can find a $\omega_X$-psh function $\varphi$ with analytic singularities, and locally, $\varphi$ is of the form  
	$$\varphi = \frac{1}{k} \log \sum |g_i|^2 + O(1),$$  
	where $\int |g_i|^2 e^{-2k\psi} < +\infty$ for some sufficiently large $k$.  
	Now, the key point is that since $S_0$ is dense in $X$ and $\varphi$ is of analytic singularity, $\varphi |_{S_0}$ is well defined.
	
	Now we calculate $\nu (\varphi |_{S_0}, x)$.
	Since $\psi \leq C f$, we know that $g_i \in \mathcal{I} (k C f)$. By Proposition \ref{listproperty}, $\mathcal{I} (k C f)$ is generated by monomials of the form $z_1^{a_1} \cdots z_n^{a_n}$ such that  
	$$a_1 +1 + \sum_{i=2}^n \frac{a_i+1}{m} > \frac{kC}{m^{\frac{n-1}{n}}}.$$
	Thus, for every $g_i$, the vanishing order of $(g_i |_{z_1=0},0)$ is at least  
	$$\sum_{i=2}^n a_i > m^{\frac{1}{n}} kC - n + 1.$$
	By Proposition \ref{listproperty}, we obtain  
	$$\nu (\varphi |_{S_0}, x) \geq m^{\frac{1}{n}} \cdot C - \frac{n-1}{k}.$$  
	The proposition follows by choosing $m$ such that $m^{\frac{1}{n}} \geq \frac{\lambda}{C}$.
\end{proof}

Let $C$ be a compact submanifold of $X$ such that there exists a Zariski open subset $C_0 \subset C$ with $\codim_C (C \setminus C_0) \geq 2$ and $C_0 \subset S_0$. Let $y_0 \in C_0$. We know that $\nu (\varphi |{S_0}, y_0) \geq \nu (\varphi |{S_0}, C_0)$. The following proposition provides a control in another direction. The proof partially follows the idea in the second part of \cite[Thm 1.2]{Ou}. The main difference is that, instead of using the algebraic construction in \cite[Section 5]{Ou} to control the singularity of the currents on $\mathbb{P}(\mathcal{N}^*_{C_0/S_0})$ near its boundary, we use Lemma \ref{controldirectimage} to directly obtain a positive current on the compact manifold $C$.

	\begin{proposition}\label{arbitrarylelong}
		Let $X$ be a compact Kähler manifold and $C \subset X$ be a compact submanifold. Let $S_0$ be a locally closed submanifold containing $C_0$, where $C_0$ is an open subset of $C$ with $\codim_C (C \setminus C_0) \geq 2$.
		
		We fix a point $y_0 \in C_0$. Let $\omega_X$ be a Kähler metric on $X$. Then there exists some constants $K_1$ and $K_2$ depending only on $(\omega_X, y_0, C_0, S_0, X)$ such that, for any $\omega_X$-psh function $\varphi$ on $X$, if $\varphi |_{S_0}$ is of algebraic singularities, then we have
		$$\nu (\varphi |_{S_0}, y_0) \leq K_1 \cdot  \nu (\varphi |_{S_0}, C_0) +K_2 .$$
	\end{proposition}
	
	 \begin{proof}
	 	Let $\pi: \widehat{X} \to X$ be the blow-up of $C$ in $X$. Let $\widehat{S}_0$ be the strict transform of $S_0$ and $D$ be the exceptional divisor of $\pi$. Let $E_0 := \widehat{S}_0 \cap D$. Then $E_0 = \mathbb{P}(\mathcal{N}^*_{C_0/S_0})$.
	 	Note that as $\codim_C (C \setminus C_0) \geq 2$, we can extend $\mathcal{N}_{C_0/S_0}$ to be a reflexive subsheaf $\mathcal{F} \subset T_{C/X}$ on $C$.  Then $E_0$ is a Zariski open subset of the compact set $\mathbb{P}(\mathcal{F}^*)$.
	 	
	 	Set $a := \nu(\varphi |_{S_0}, C_0)$. We consider the current $T := \pi^* (\omega_X + \ddbar \varphi) \geq 0$ on $\widehat{X}$. By construction, the restriction $T |_{\widehat{S}_0}$ is well defined. By Proposition \ref{listproperty}, we have
	 	$$ a = \nu(T |_{\widehat{S}_0}, E_0). $$
	 	Now we set $T_a := T |_{\widehat{S}_0} - a [E_0]$. Then $T_a$ is still a positive current on $\widehat{S}_0$, and $T_a |_{E_0}$ is a well-defined positive current on $E_0$. Here we use the fact that $\varphi |_{S_0}$ has algebraic singularities.  
	 	
	 	Let $E$ be a desingularisation of $\mathbb{P}(\mathcal{F}^*)$. Note that $\codim_E (E \setminus E_0)$ might be 1, and the singularity of $T_a |_{E_0}$ might be transcendental near the boundary because of the transcendentality of $S_0$. Therefore, we cannot à priori extend $T_a |_{E_0}$ to be a current on $E$.
	 	
	 	Let $\pi_{E_0}: E_0 = \mathbb{P}(\mathcal{N}^*_{C_0/S_0}) \to C_0$. Let $r$ be the dimension of the generic fiber of $\pi_{E_0}$, and let $\omega_{\widehat{X}}$ be a Kähler metric on $\widehat{X}$. After multiplying by some constant, we can assume that $(\pi_{E_0})_* \left( (\omega_{\widehat{X}} |_{E_0})^r \right) = 1$ at $y_0$. We consider the push-forward
	 	$$ T_{C_0} := (\pi_{E_0})_* \left( T_a |_{E_0} \wedge (\omega_{\widehat{X}} |_{E_0})^r \right). $$
	 	By construction, the $(1,1)$-current $T_{C_0} \geq 0$ on $C_0$. As $\codim_C (C \setminus C_0) \geq 2$, by Proposition \ref{listproperty}, we can extend it to be a positive current on $C$, denoted by $T_C$.
	 	
	 	Thanks to Proposition \ref{controldirectimage}, we have
	 	$$ \nu(T_C, y_0) \geq \nu(T |_{\widehat{S}_0}, \pi_{E_0}^{-1}(y_0)) - a = \nu(\varphi |_{S_0}, y_0) - a, $$
	 	where the second equality comes from Proposition \ref{listproperty}. 
	 	Note that by Lemma \ref{upperbound} below, the cohomology class of $T_C$ is bounded by $ (a K_1 +K_2)  \cdot c_1(\omega_X |_C )$ for some uniform constants $K_1$ and $K_2$. Then $\nu(T_C, y_0)$ is uniformly upper-bounded by $ K \cdot (a K_1 +K_2)$, where $K$ is a constant depending on $c_1 (\omega_X |_C)$ and the Seshadri constant of $y_0$ in $C$.  Therefore, we have
	 	$$ K \cdot ( a K_1 + K_2)  \geq \nu(T_C, y_0) \geq \nu(\varphi |_{S_0}, y_0) - a, $$ 
	 	and the proposition is proved.
	 \end{proof}

		\begin{lemma}\label{upperbound}
			Let $T_C$ be the current on $C$ constructed in the proof of Proposition \ref{arbitrarylelong}. Then there exists a uniform constants $K_1$ and $K_2$ independent of $a$, such that $(a K_1  +K_2 ) \cdot c_1 (\omega_X |_C) - T_C$ is a psef class.
		\end{lemma}
		
		\begin{proof}
			Let $\alpha := c_1 (\pi^* \omega_X ) - a \cdot c_1 (\mathcal{O}_{\widehat{X}} (D)) \in H^{1,1} (\widehat{X}, \mathbb R)$. Let $\rho : E \to\mathbb P (\mathcal F^*) $ be a desingularization and $\pi_E : E \to C$ be the projection. Consider the class
			$$
			\beta :=(\pi_E)_* \Big(\rho^* 
			\big(\big(\alpha  \wedge  (\omega_{\widehat{X}} )^r \big)|_{\mathbb P (\mathcal F^*)} \big)\Big) .
			$$
			Note that  $\beta$ is a $(1,1)$-class on $C$ and is upper bounded by $(a  K_1 +K_2) \cdot c_1 (\omega_X |_C)$ for some uniform constants $K_1$ and $K_2$ independent of $a$. To prove the lemma, it suffices to show that 
			\begin{equation}\label{neeprove}
				c_1 (\beta) = c_1 (T_C) \qquad\text{on } C.
			\end{equation}
			
			Let $s_D \in H^0 (\widehat{X}, \mathcal{O}_{\widehat{X}} (D))$ be the canonical section (which vanishes along $D$), and fix a smooth metric $h_D$ on $\mathcal{O}_{\widehat{X}} (D)$. Consider the current 
			$$
			T_{\widehat{X}}:=\pi^* \omega_X +\ddbar \varphi - a \ddbar \log |s_D|^2_{h_D} - a i\Theta_{h_D} (\mathcal{O}_{\widehat{X}} (D)) \qquad\text{on } \widehat{X}.
			$$
			Then we have $c_1 (T_{\widehat{X}}) = \alpha$ and $T_{\widehat{X}} |_{\widehat{S}_0} =T_a$ on  $\widehat{S}_0$. 
			
			Set $A:= \pi^* \omega_X - a i\Theta_{h_D} (\mathcal{O}_{\widehat{X}} (D))$. Then $A$ is a smooth form in the class of $\alpha$, and we have
			$$
			A |_{E_0} = T_a |_{E_0} + \ddbar \psi \qquad\text{on } E_0,
			$$ 
			where $\psi = ((\varphi - a \log |s_D|^2_{h_D} ) |_{\widehat{S}_0} ) |_{E_0}$ is well-defined on $E_0$. As a consequence,
			$$
			A |_{E_0} \wedge \big(\omega_{\widehat{X}} |_{E}\big)^r = T_a |_{E_0} \wedge \big(\omega_{\widehat{X}} |_{E}\big)^r + \ddbar (\psi \cdot \big(\omega_{\widehat{X}} |_{E}\big)^r) \qquad\text{on } E_0.
			$$
			Therefore, we have 
			\begin{equation}\label{current}
				p_* (A |_{E_0} \wedge \big(\omega_{\widehat{X}} |_{E}\big)^r) = T_{C_0} + \ddbar \Phi \qquad\text{on } C_0
			\end{equation}
			for some function $\Phi$ defined on $C_0$. Since $\codim_C (C \setminus C_0) \geq 2$, by Proposition \ref{listproperty}, the equality \eqref{current} extends to the whole space $C$. Note that the class of the extended current on the left-hand side of \eqref{current} is precisely $\beta$. This completes the proof of \eqref{neeprove}.
		\end{proof}

		\medskip
		
		Now we are going to prove Theorem \ref{mainthm}. In the setting of Theorem \ref{mainthm}, we fix a point $y_0 \in C_0$. If $M$ is smooth, then by applying Proposition \ref{lelong} to $(M, S_0, y_0)$, we can construct a $\omega_X$-psh function $\varphi$ on $M$ such that $\nu (\varphi|_{S_0}, y_0)$ is arbitrarily large. 
		By applying the extension theorem \cite[Thm 1.1]{CT}, we obtain a $\omega_X$-psh function $\varphi$ on $X$ such that $\nu (\varphi|_{S_0}, y_0)$ is arbitrarily large. Together with Proposition \ref{arbitrarylelong}, we conclude that $\nu (\varphi|_{S_0}, C_0)$ is arbitrarily large. This implies that $\mathcal{N}^* _{C_0/S_0}$ is psef (cf. the proof of Theorem \ref{mainthm}). In the general case, when $M$ is not necessarily smooth, we have the following proposition.
		
		\begin{proposition}\label{addlemma}
			In the setting of Theorem \ref{mainthm}, let $\pi: \widehat{X} \to X$ be the blow-up of $C$ in $X$. Let $\widehat{S}_0$ be the strict transform of $S_0$, and let $D$ be the exceptional divisor of $\pi$. Define $E_0 := \widehat{S}_0 \cap D$. Let $\omega_{\widehat{X}}$ be a Kähler metric on $\widehat{X}$.  
			
			If $\dim M > \dim S_0$, then for any $m \in\mathbb{N}$, there exists a $\omega_{\widehat{X}}$-psh function $\varphi_m$ on $\widehat{X}$ such that $\varphi_m |_{\widehat{S}_0}$ has analytic singularities and $\nu (\varphi_m |_{\widehat{S}_0}, E_0) \geq m$.  
		\end{proposition}
		
		\begin{proof}
			Let $\widehat{M} \subset \widehat{X}$ be the strict transform of $M$. We desingularize the singular locus of $\widehat{M}$ and obtain a bimeromorphism $p: \widetilde{X} \to \widehat{X}$. Let $\widehat{S}_0$ be the strict transform of $S_0$ in $\widehat{X}$. Since $S_0$ is dense in $M$, $\widehat{S}_0$ is not contained in the singular locus of $\widehat{M}$. Let $\widetilde{S}_0 \subset \widetilde{X}$ be the strict transform of $\widehat{S}_0$. Define $\widetilde{E}_0 \subset \widetilde{S}_0$ as the unique component of $p^{-1} (E_0) \cap \widetilde{S}_0$ such that $p (\widetilde{E}_0) = E_0$. The uniqueness follows from the fact that $E_0$ has codimension $1$ in $\widehat{S}_0$ and that $\widehat{S}_0$ is smooth. We have the following diagram:
			
			\[
			\begin{tikzcd}
				\widetilde{y}_0\in 	\widetilde{E}_0\arrow[r, hookrightarrow] \arrow[d, "p"] & \widetilde{S}_0 \arrow[r, hookrightarrow] \arrow[d, "p"] & \widetilde{X}\arrow[d, "p"] \\
				\widehat{y}_0\in 	E_0 \arrow[r, hookrightarrow] \arrow[d, "\pi"] & \widehat{S}_0 \arrow[r, hookrightarrow] \arrow[d, "\pi"] & \widehat{X} \arrow[d, "\pi"] \\
				y_0\in C_0 \arrow[r, hookrightarrow] & S_0 \arrow[r, hookrightarrow] & X
			\end{tikzcd}
			\]
			
			Let $\widetilde{y}_0 \in \widetilde{E}_0$ be a generic point. Since $E_0 \subset \widehat{S}_0$ has codimension one, the bimeromorphism $p: \widetilde{S}_0 \to \widehat{S}_0$ is an isomorphism in a neighborhood of the generic point $\widetilde{y}_0$. In particular, $\widetilde{S}_0$ is smooth in a neighborhood of $\widetilde{y}_0$.  Set $\widehat{y}_0 := p (\widetilde{y}_0)$ and $y_0 := \pi (\widehat{y}_0)$.  
			Now, we fix a Kähler metric $\omega_X$ on $X$. By using \cite{DP}, there exist Kähler metrics $\omega_{\widetilde{X}}$ on $\widetilde{X}$ and $\omega_{\widehat{X}}$ on $\widehat{X}$ such that
			\[
			p_* (c_1 (\omega_{\widetilde{X}})) = c_1 (\omega_{\widehat{X}}) \quad \text{and} \quad \pi_* (c_1 (\omega_{\widehat{X}})) = c_1 (\omega_X).
			\]
			
			Since $\widetilde{M}$ is smooth and $\widetilde{S}_0$ is smooth in a neighborhood of $\widetilde{y}_0$, we can apply Proposition \ref{lelong} to $(\widetilde{M}, \widetilde{S}_0, \widetilde{y}_0)$.  
			Together with the extension theorem \cite[Thm 1.1]{CT}, for any $m > 0$, we can find a $\omega_{\widetilde{X}}$-psh function $\widetilde{\varphi}_m$ on $\widetilde{X}$ such that $\widetilde{\varphi}_m |_{\widetilde{S}_0}$ has analytic singularities and  
			\[
			\nu (\widetilde{\varphi}_m |_{\widetilde{S}_0}, \widetilde{y}_0) \geq m.
			\]
			Set $\widetilde{T}_m := \omega_{\widetilde{X}} +\ddbar \widetilde{\varphi}_m \geq 0$. Then  
			\[
			\nu (\widetilde{T}_m |_{\widetilde{S}_0},  \widetilde{y}_0) \geq m.
			\]
			
			We consider the push-forward  
			\[
			\widehat{T}_m := p_* (\widetilde{T}_m), \quad T_m := (p\circ\pi)_* \widetilde{T}_m.
			\]
			Then  
			\[
			c_1 (\widehat{T}_m) = c_1 (\omega_{\widehat{X}}), \quad c_1 (T_m) = c_1 (\omega_X),
			\]
			and the restrictions $\widehat{T}_m |_{\widehat{S}_0}$ and $T_m |_{S_0}$ are well defined.  
			By using Lemma \ref{comparelelong}, there exists a constant $C_1$ independent of $m$ such that  
			\[
			\left| \nu (\widehat{T}_m |_{\widehat{S}_0},  \widehat{y}_0) -\nu (\widetilde{T}_m |_{\widetilde{S}_0},  \widetilde{y}_0) \right| \leq C_1.
			\]
			Therefore, we have  
			\[
			\nu (\widehat{T}_m |_{\widehat{S}_0},  \widehat{y}_0) \geq m - C_1.
			\] 
			
			There are two possibilities:
			
			- If $\nu (\widehat{T}_m |_{\widehat{S}_0},  E_0) \geq \frac{m}{2}$ for infinitely many $m \in \mathbb{N}^*$, then the proposition is proved.
			
			- If $\nu (\widehat{T}_m |_{\widehat{S}_0},  E_0) \leq \frac{m}{2}$ for every sufficiently large $m$, then, thanks to Lemma \ref{controllelong}, we have  
			\[
			\nu (T_m |_{S_0}, y_0) \geq\frac{1}{K} \cdot \left(\nu (\widehat{T}_m |_{\widehat{S}_0},  \widehat{y}_0) - \nu (\widehat{T}_m |_{\widehat{S}_0},  E_0) \right) \geq  \frac{1}{K} \cdot \left(\frac{m}{2}-C_1\right)
			\]
			for some uniform constant $K$ independent of $m$. By applying Proposition \ref{arbitrarylelong}, we obtain  
			\[
			\nu (T_m |_{S_0}, C_0) \geq \frac{\nu (T_m |_{S_0}, y_0) -K_2}{K_1}  \geq  \widetilde{K}_1 m - \widetilde{K}_2 .
			\]
			for some uniform constants $\widetilde{K}_1 >0$ and $\widetilde{K}_2$ independent of $m$.
			We write $T_m = \omega_X +\ddbar \varphi_m$ on $X$. Then $\pi^*\varphi_m$ satisfies the proposition.
		\end{proof}

		\begin{lemma}\label{comparelelong}
			There exists a constant $C_1$ independent of $m$ such that 
			\[
			|\nu (\widehat{T} |_{\widehat{S}_0},  \widehat{y}_0) -\nu (\widetilde{T} |_{\widetilde{S}_0},  \widetilde{y}_0)| \leq C_1.
			\]
		\end{lemma}
		
		\begin{proof}
			Let $\sum D_i$ be the exceptional divisor of $p$. By construction, we have  
			\[
			c_1 (\omega_{\widetilde{X}})= p^* c_1 (\omega_{\widehat{X}}) +\sum a_i [D_i]
			\]
			for some constants $a_i$. 
			Then, we can write $\widetilde{T} =\omega_{\widetilde{X}} + \ddbar \widetilde{\varphi}$ and $\widehat{T} =\omega_{\widehat{X}} + \ddbar \psi$ for some function $\psi$ on $\widehat{X}$.
			Since $c_1 (p_* (\omega_{\widetilde{X}})) =c_1 (\omega_{\widehat{X}} )$, there exists a function $\Phi$ on $\widehat{X}$ such that  
			\[
			p_* (\omega_{\widetilde{X}}) =\omega_{\widehat{X}} +\ddbar \Phi.
			\]
			As a consequence, we have the pointwise identity 
			\[
			p_* ( \widetilde{\varphi}) = \psi -\Phi +C_0 \qquad \text{on } \widehat{X} 
			\]
			for some constant $C_0$.
			
			Since $p |_{\widetilde{S}_0}$ is an isomorphism near $\widetilde{y}_0$, we can find a neighborhood $\widetilde{U} \subset \widetilde{S}_0$ of $\widetilde{y}_0$ and a neighborhood $\widehat{U} \subset \widehat{S}_0$ of $\widehat{y}_0$ such that
			$p|_{\widetilde{U}}:  \widetilde{U} \to \widehat{U}$ is an isomorphism. Consequently, at the level of quasi-psh functions, we have the identity  
			\[
			(\psi -\Phi +C_0) \circ (p|_{\widetilde{U}}) =\widetilde{\varphi}  \qquad \text{on } \widetilde{U}.
			\]
			Thus, we obtain 
			\[
			\nu ((\psi -\Phi +C_0)|_{\widehat{S}_0}, \widehat{y}_0 )= \nu (\widetilde{\varphi}  |_{\widetilde{S}_0}, \widetilde{y}_0).
			\]
			Since $\Phi$ depends only on the blow-up and the coefficients $a_i$, and is well-defined on $\widehat{S}_0$, we conclude that  
			\[
			|\nu (\psi |_{\widehat{S}_0}, \widehat{y}_0 )- \nu (\widetilde{\varphi}  |_{\widetilde{S}_0}, \widetilde{y}_0)| \leq C_1
			\]
			for some uniform constant $C_1$.
		\end{proof}
		
Now we can prove Theorem \ref{mainthm}:
\begin{proof}[Proof of  Theorem \ref{mainthm}]
	Let $\pi: \widehat{X} \to X$ be the blow-up of $C$ in $X$. Let $\widehat{S}_0$ be the strict transform of $S_0$, and let $D$ be the exceptional divisor of $\pi$. Define $E_0 := \widehat{S}_0 \cap D$.
	Then we have $E_0 = \mathbb{P} (\mathcal{N}^*_{C_0/S_0})$.
	
	By applying Proposition \ref{addlemma}, for every $m\in\mathbb{N}$, there exists a $\omega_{\widehat{X}}$-psh function $\varphi_m$ such that $a_m := \nu (\varphi_m |_{\widehat{S}_0}, E_0) \geq m$, where $\omega_{\widehat{X}}$ is a Kähler metric on $\widehat{X}$.
	
	Consider the current $T_m := \omega_{\widehat{X}}+ \ddbar \varphi_m \geq 0$ on $\widehat{X}$. By construction, the restriction $T_m |_{\widehat{S}_0}$ is well defined. Now, we set 
	\[
	G_m := T_m |_{\widehat{S}_0}  -a_m [E_0].
	\] 
	Then, $G_m |_{E_0}$ is a well-defined positive current on $E_0$, and we have 
	\begin{equation}
		c_1 (G_m |_{E_0}) =  c_1 ( \omega_{\widehat{X}}  |_{E_0}  + a_m  \cdot \mathcal{O}_{\mathbb{P} (\mathcal{N}^*_{C_0/S_0})} (1) ) \qquad\text{on } E_0.
	\end{equation}
	Since $a_m \geq m$ and $G_m |_{E_0} \geq 0$, we conclude that $ \mathcal{N}^*_{C_0/S_0}$ is psef in the sense of Definition \ref{anapsf}. This completes the proof of the first part of the theorem.
	
	\medskip
	
	For the second part of the theorem comes from the standard argument. We consider the diagonal $C\simeq X \subset X\times X$ and  let $S_0 \subset X\times X$ be the ananlytic graphe of $\cF$. Then $\mathcal{N}^* _{C/S_0}$ is not psef on the locally free locus of $\cF$. By the first part of the theorem, $S_0$ has the same dimension as its Zariski closure in $X\times X$.  It will imply that $\cF$ is induced a meromorphic map by the arguments in \cite[section 4.1]{CP19} (or \cite[Thm 1.4]{Ou} ). 
\end{proof}

%%%%%%%%%%%%%%%%%%%%%%%%%%%%%%%%%%%%%%%%%%%%%%%%%%%%%%%

\section{Algebraicity of positive foliations.}

Let $X$ be a compact K\"ahler manifold, and let $\cG \subset T_X$ be a saturated holomorphic subsheaf. We say that $\cG$ is a holomorphic foliation on $X$ if it is closed under the Lie bracket, namely, $[\cG, \cG] \subset \cG$. It is said to be algebraic if its leaves are algebraic. Equivalently, this means that there exists a meromorphic fibration $f: X \dasharrow Y$ with algebraic leaves such that 
$\cG = \ker(df)$ generically on $X$. 
\medskip

\noindent {\bf Convention}. The notions of pseudo-effectivity, mobile classes, slopes, and numerical dimension will be recalled in the Appendix, which also provides their basic properties along with references for the proofs.
\smallskip

The main result of this section is as follows. It extends \cite[Thm 1.1, Thm 1.2]{CP19} to the K\"ahler setting.

\begin{theorem}\label{main} 
	Let $X$ be a compact Kähler manifold, and let $\cF \subset T_X$ be a holomorphic foliation such that $\mu_{\alpha}(\cF) > 0$ for some $\alpha \in \Mov(X)$. 
	Let $\cG$ be the maximal $\alpha$-destabilizing subsheaf of $\cF$ (cf. Theorem \ref{max} in the appendix). 
	
	\begin{itemize}
		\item [(1)] Then $\cG$ is an algebraic foliation.
		
		\item [(2)] Let $p: X \dasharrow Y$ be the meromorphic fibration induced by $\cG$, and let $X_y$ be a generic fiber of $p$. Then $X_y$ is rationally connected. 
		
	\end{itemize}
	In particular, $X$ is uniruled, and $K_X$ is not pseudo-effective.
\end{theorem}

\begin{proof} 
	Since $H^2 (X, \mathbb{Q})$ is dense in $H^2 (X, \mathbb{R})$, thanks to the Hodge decomposition, there exists a Kähler metric $\omega$ and a holomorphic $(2,0)$-form $\sigma \in H^{2,0} (X)$ such that the cohomology class of $\omega + \sigma + \overline{\sigma}$ is rational; see \cite{C} or \cite[8.3]{Ou}, for more details. For example, if $X$ is projective, we can take $\sigma = 0$.
	
	Since $\cG$ be the maximal $\alpha$-destabilizing subsheaf of $\cF$ and $\mu_{\alpha}(\cF) > 0$, we have $\mu_{\alpha}(\cG) > 0$. By applying Corollary \ref{foliation} in the appendix, we infer that $\cG$ is a foliation and  
	\begin{equation}\label{positve}
		\mu_{\alpha, \min} (\cG)  > 0.
	\end{equation}  
	By Proposition \ref{slope}, $\cG^\star$ is not psef. Then Theorem \ref{mainthm} or \cite[Thm  1.4 ]{Ou} implies that the foliation $\cG$ is induced by a meromorphic fibration  $p: X\mathrel{-\,}\rightarrow Y$, whose relative tangent bundle is equal to $\cG$. 
	
	\medskip
	
\noindent We next discuss the second point (2).  Let $X_y$ be a generic fiber of $p$.
We claim that the restriction is identically zero:
\begin{equation}\label{res}
	\sigma |_{X_y} = 0
	\end{equation} 
	The argument is as follows: the restriction 
$\sigma|_\cG$ can be seen as section
\[\sigma|_\cG\in H^0(X, \Hom(\Lambda^{[2]}\cG, \mathcal O_X)),\]
where $\Lambda^{[2]}\cG$ is the reflexive hull of the second exterior power of the sheaf $\cG$.  Since
$\mu_{\alpha, \min}(\Lambda^{[2]}\cG) > 0$ by Theorem \ref{sum}, the $H^0$ above vanishes, which proves our claim.
\smallskip

 As in \cite{CP19}, we consider birational maps $\pi_X: \wh X\to X$, $ \pi_Y: \wh Y\to Y $, and
	$\wh p: \wh X\to \wh Y$, so that $p\circ \pi_X= \pi_Y\circ\wh p$. 
	Moreover, we can assume that the following statements are satisfied.
	\begin{itemize}
		\smallskip
		
		\item The discriminant locus of $\wh p$ is a divisor $E$ with snc support
		\smallskip
		
		\item If a component of $\wh p^{-1}(E)$ is exceptional, then it is also $\pi_X$-exceptional. 	
	\end{itemize}
	
	\noindent The kernel of $p$ i.e. $\cG= \ker(dp)$ induces a foliation $\wh \cG$
	on $\wh X$, and it is known that $\det(\wh \cG^\star)= K_{\wh X/\wh Y}- D(\wh p)$, cf. for example \cite[Section 2]{Dru1}.

	Let $\wh X_y$ be the generic fiber of $\wh p$. We first prove that  $K_{\wh X_y}$ is not psef.  We assume by contradiction that $K_{\wh X_y}$ is psef. Since the cohomology class of $\omega+ \sigma +\overline{\sigma}$ is rational and $\sigma |_{\wh X_y}=0$ by \eqref{res},  we can apply Theorem \ref{thm1} 
	to $\wh p$.  Then Theorem \ref{thm1} shows that $\det(\wh \cG^\star)= K_{\wh X/\wh Y}- D(\wh p)$ is psef.
	Then $\mu_{\pi_X^\star\alpha}(\wh \cG^\star)\geq 0$.  On the other hand,  by \eqref{positve}, we have
	$\mu_{\pi_X^\star\alpha}(\wh \cG^\star)=-\mu_{\alpha}(\cG)<0$, 
	where the first equation comes from the fact that $\det(\wh \cG^\star)$ and $\pi_X^\star\det(\cG^\star)$ can only differ by some $\pi_X$-exceptional divisor. We obtain thus a contradiction. Therefore $K_{\wh X_y}$ is not psef.
	
	Our goal is to prove that $\wh X_y$ is rational connected. Note that the cohomology class of $\omega+ \sigma +\overline{\sigma}$ is rational and $\sigma |_{\wh X_y}=0$ by \eqref{res}. 
	Then $\wh X_y$ is projective. As $K_{\wh X_y}$ is not psef, by using \cite{BDPP}, we know that $\wh X_y$ is uniruled.  Now we can follow the proof of  \cite[Thm 4.7]{CP19}. 
	
Consider the relative MRC fibration of $\wh p$:  $r: \wh X \dasharrow \wh Z$, so that we also have a map $\wh s: \wh Z \dasharrow \wh Y$.
\footnote{Since the existence of a relative rational quotient (or MRC) rests on the compactness properties of the Chow-Barlet space, it remains valid in the K\"ahler context.} By replacing $\wh X$ and $\wh Z$ with birational models (for which we will use the same notation), we can assume that 
the maps $r$ and $s$ are regular.	
We denote by $\widehat{\mathcal{H}} := \ker (d\wh s)$, the foliation induced by the kernel of $\wh s$.
	
Next we claim that $\det (\widehat{\mathcal{H}})^\star$ is psef. To see this, let $k$ be the dimension of 
the generic fiber of $r$. We consider the rational $(1,1)$-class $r_\star \big(   (\omega+\sigma +\overline{\sigma})^{k+1}\big)$ on  $\wh Z$.
	Since the fibers of $r$ are rational connected, the restriction of $\sigma$ vanishes. Then we have 
	$$r_\star (\omega+\sigma +\overline{\sigma})^{k+1} = r_\star\omega^{k+1} +k\cdot  r_\star  \big (  \omega^{k} \wedge \sigma +  \omega^{k} \wedge \overline{\sigma}\big) ,$$
	where the first term of RHS is $(1,1)$-strict postive over a generic fiber of $\wh Z \to \wh Y$ and the second term of RHS is of type $(2,0) + (0,2)$. Morevoer, thanks to \eqref{res}, the restriction of $r_\star  \big (  \omega^{k} \wedge \sigma +  \omega^{k} \wedge \overline{\sigma}\big) $ on the generic fiber of $\wh s$ vanishes.
	Then we can apply Theorem \ref{thm3} to $\wh Z \to \wh Y$ with respect to $ r_\star  \big(   \omega^{k+1}\big) +k\cdot  r_\star  \big (  \omega^{k} \wedge \sigma +  \omega^{k} \wedge \overline{\sigma}\big)$. 
Actually the context here is a bit more general than in Theorem \ref{thm3}, but the point is that the $(1,1)$ current 
\[r_\star \omega^{k+1}\]
is closed, greater than a Kähler metric on $\wh Z$ and smooth when restricted to the $\wh s$-inverse image of the complement of a Zariski closed subset of $\wh Y$. The local $L^{\frac{2}{k}}$-extension in Step four still works when $m_0$ is large enough. Then the proof of Theorem \ref{thm3} goes through.	
In particular, it follows that $\det (\widehat{\mathcal{H}})^\star$ is psef. 
	
	  On the other hand, we have a natural generic surjective morphism $\wh \cG \to r^*\widehat{\mathcal{H}}$. Thanks to \eqref{positve}, we know that $\mu_{\pi^*\alpha} (r^*\widehat{\mathcal{H}} )>0$ if $\widehat{\mathcal{H}}$ is non zero. Therefore $\widehat{\mathcal{H}}$ is zero and $\wh X_y$ is rationally connected.
\end{proof}

\medskip

As an application of Theorem \ref{main}, we obtain the following extension of \cite[Thm 1.2]{CP19} \cite{CPet} to the Kähler case.

\begin{cor}\label{tens}
	
	Let $X$ be a compact Kähler manifold, and assume that $K_X$ is pseudo-effective. Let $\cQ$ be a torsion-free quotient of $\otimes^m \Omega^1_X$ for some $m > 0$. Then $\det(\cQ)$ is pseudo-effective. 
	
	In particular, if $L \to \otimes^m\Omega^1_X$ is an injective morphism of sheaves and $L$ is a line bundle on $X$, then $\nu(X,L) \leq \nu(X,K_X)$, where $\nu(X,L)$ denotes the numerical dimension of $L$. 
	
\end{cor}

\begin{proof} 
	Assume, by contradiction, that $\det \cQ$ is not pseudo-effective. Then there exists a class $\alpha \in \text{Mov} (X)$ such that $c_1 (\det \cQ) \cdot \alpha < 0$.  
	By passing to duality, this implies that $\mu_{\alpha, \max}(\otimes^m T_X) > 0$.  
	Together with Theorem \ref{sum}, we obtain  
	$$\mu_{\alpha, \max}(T_X) = \frac{1}{m} \mu_{\alpha, \max}(\otimes^m T_X) > 0.$$  
	
	Thanks to Theorem \ref{max}, there exists a unique maximal $\alpha$-destabilizing subsheaf $\cF \subset T_X$. Then $\mu_\alpha (\cF) = \mu_{\alpha, \max}(T_X) > 0$.  
	By Corollary \ref{foliation}, $\cF$ is a foliation. Applying Theorem \ref{main} to $\cF$, we conclude that $K_X$ is not pseudo-effective, leading to a contradiction.  
\medskip
	
\noindent For the second part, we argue as follows. Let $\cQ$ be the quotient of $\otimes^m\Omega^1_X$ by $L$. Even if $\cQ$ might have torsion, the determinant $\det \cQ$ is still psef. On the other hand, we have the equality
\[c(m)c_1(K_X)= c_1(L)+ c_1(\cQ)\]
for some positive integer $c(m)$. The conclusion follows by Lemma \ref{ndim} in the Appendix .
\end{proof}

\begin{cor}\label{unricri}
	Let $X$ be a compact Kähler manifold. The following statements are equivalent:
	
	1. $X$ is uniruled.
	
	2. $\mu_{\alpha, \max}(T_X) > 0$ for some $\alpha \in \Mov(X)$. 
	
	3. $\mu_{\alpha}(T_X) > 0$ for some $\alpha \in \Mov (X)$.
\end{cor}

Note that the equivalence of (1) and (3) is proved differently in \cite[Thm 1.1]{Ou}. Here, we additionally prove that (2) implies (3).

\begin{proof}
	By definition, (3) implies (2).  
	
	For (2) $\Rightarrow$ (1), using Theorem \ref{max} and Corollary \ref{foliation}, we know that the maximal $\alpha$-destabilizing subsheaf $\cF \subset T_X$ is a foliation and satisfies $\mu_\alpha (\cF) > 0$. By Theorem \ref{main}, $X$ is uniruled. 
	
	For (1) $\Rightarrow$ (3), if $X$ is uniruled, then $X$ can be covered by a family of free rational curves $(C_t)_{t\in \cC}$. Then $[C_t] \cdot L \geq 0$ for any pseudo-effective class $L \in H^{1,1} (X, \mathbb R)$. Therefore we have $[C_t] \in \Mov (X)$. Furthermore,  
	$$\mu_{[C_t]} (T_X) = -K_X \cdot [C_t] > 0,$$  
	since $C_t$ is a rational curve.   
\end{proof}	

Now we extend  \cite[Thm 4.7]{CP19} to the Kähler setting.

\begin{cor}\label{folRC} 
	Let $X$ be a compact Kähler manifold. Let $\cF \subset T_X$ be a holomorphic foliation. If $\mu_{\alpha, \min}(\cF) > 0$ for some $\alpha \in \Mov (X)$, 
	then $\cF$ is  algebraic and its leaves are rationally connected.
\end{cor}

\begin{proof}
The proof follows almost the same argument as that in Theorem \ref{main}. We only sketch the proof.
	By using Proposition \ref{slope} and Theorem \ref{mainthm}, we know that $\cF$ is algebraic.
	Let $p: X\dasharrow Y$ be the meromorphic fibration induced by $\cF$.
	By applying Theorem \ref{main} to $\cF$, there exists a subfoliation $\cG \subset \cF$ such that its leaves are rationally connected. In particular, the generic fiber of $p$ is uniruled.
	Let $\sigma \in H^{2,0} (X)$ be a non-zero holomorphic $2$-form on $X$, and let $X_y$ be a generic fiber of $p$.
By the same argument as in Theorem   \ref{main}, we know that $\sigma |_{X_y} = 0$ is identically zero. 
%In fact, by passing to the dual, the inclusion $\cF \subset T_X$ induces a morphism $\pi: \Omega^2_X \to(\Lambda^2 \cF^\star)^{\star\star }$.
%	Then, $\pi (\sigma) \in H^0 (X, (\Lambda^2 \cF^\star)^{\star\star })$. Since $\mu_{\alpha, \min}(\cF) > 0$, we also have $\mu_{\alpha, \min}((\Lambda^2 \cF^\star)^{\star\star }) > 0$ by Theorem \ref{sum}.
%	Thus, there is no effective rank-one subsheaf in $(\Lambda^2 \cF^\star)^{\star\star }$.
%	As a consequence, $\pi (\sigma) = 0 \in H^0 (X, (\Lambda^2 \cF^\star)^{\star\star })$. This implies that $\sigma (v_1, v_2) = 0$ for any two vectors $v_1, v_2 \in \cF |_{X_y} = T_{X_y}$.
%	Thus, $\sigma |_{X_y} = 0$.
	
	\medskip
	
	We consider the relative MRC fibration of $p$, denoted by $r: X\dasharrow Z$. We also have a natural map $s : Z \dasharrow Y$.
	By replacing $X, Z,$ and $Y$ with suitable birational models (for which we keep the same notation), we can assume that the maps $r$ and $s$ are regular.
	
	By construction, $Z_y$ is not uniruled. Then, $K_{Z_y}$ is psef by Corollary \ref{unricri}.
	We denote by $\mathcal{H} := \ker (d s)$ the foliation induced by the kernel of $s$.
	The term $r_* ((\omega+\sigma +\overline{\sigma})^{k+1})$ is a rational class, and its $(1,1)$-part has strict positivity on the generic fiber of $s$.
	Since $\sigma |_{X_y} = 0$, the $(2,0)$-part of $r_* ((\omega+\sigma +\overline{\sigma})^{k+1})$ vanishes on the generic fiber of $s$.
	Thus, we can apply Theorem \ref{thm3}. The psefness of $K_{Z_y}$ implies that $\det \mathcal{H}^\star$ is psef.
	Recall that we have a generic surjective morphism $\cF \to r^\star \mathcal H$.
	Since $\mu_{\alpha, \min}(\cF) > 0$, the psefness of $\det \mathcal{H}^\star$ implies that $\mathcal{H}$ is zero.
	Thus, the generic fiber of $p$ is rationally connected.
\end{proof}

\begin{cor} Let $p:\cX\to \mathbb D$ be a proper holomorphic submersion, where $\cX$ is a K\"ahler manifold. If the canonical bundle of the central fiber $\cX_0= p^{-1}(0)$ is pseudo-effective, then so is the canonical bundle of $\cX_t$ for any $t\in \mathbb D$.  
\end{cor}

\noindent This result is a particular case of the conjecture of Siu, about the K\"ahler case of the invariance of plurigenera.

\begin{proof} This is an immediate consequence of either \cite{Fuj}, or \cite{Lev} which show that if one fibre of $p$ is uniruled, then so are all. 
\end{proof}

\smallskip

\noindent To finish this section, we propose the next question.

\begin{question}\label{q} Let $q:X\to Q$ be the rational quotient (or MRC) of $X$. Let $m\in\mathbb N^*$ and let $L$ be a psef line bundle on $X$ with an inective morphism of sheaves $L\to \otimes^m \Omega^1_X$.   Then $\nu(X,L)\leq \nu(Q,K_Q)$.
\end{question}

%\begin{cor}\label{nu} Let $q:X\to Q$ be the rational quotient (or MRC) of $X$. Assume that Question \ref{q} has a positive answer. Let $m\in\mathbb N^*$ and let $L$ be a psef line bundle on $X$ with an inective morphism of sheaves $L\to \otimes^m \Omega^1_X$.   Then $\nu(X,L)\leq \nu(Q,K_Q)$.
%\end{cor}

%\begin{proof} Note that $q_*(\otimes^m\Omega^1_X) =\otimes^m\Omega^1_Q$  over the complement of a the discriminant divisor on $Q$. If the answer of Question \ref{q} is positive, we have  $q_*(\otimes^m\Omega^1_X) =\otimes^m\Omega^1_Q$  on $Q$. 
%Since $L\to \otimes^m \Omega^1_X$ and $L$ is psef, we thus have $q_*(L)\subset \otimes^m\Omega^1_Q$ and $L=q^*(q_*(L))$.  Therefore $\nu(X,L)=\nu(Q, f_*(L))\leq \nu(Q,K_Q)$ by Corollary \ref{tens}.
%\end{proof}

%%%%%%%%%%%%%%%%%%%%%%%%%%%%%%%%%%%%
%%%%%%%%%%%%%%%%%%%%%%%%%%%%%%%%%%%%

\section{Appendix: stability with respect to the degree induced by a Gauduchon metric}	

\noindent In this section we collect a few results concerning the semi-stability with respect to a Gauduchon metric. Actually the theory is perfectly similar to the "classical" one, in which one uses a K\"ahler metric for the definition of the degree function.
\smallskip

\noindent Let $X$ be a $n$-dimensional compact complex manifold and let $g$ be a Hermitian metric on $X$. We say that $g$ is a Gauduchon metric, if $\ddbar g^{n-1}=0$. Recall that given any Hermitian metric on a compact complex manifold, its conformal class contains a unique Gauduchon metric, as proved in \cite{Gau}.  A first result we recall here is the following criteria due to A. Lamari, which fits perfectly with our context.

\begin{lemma}\label{dual}\cite{Lam} Let $\gamma$ be a real $(1,1)$ form on $X$, such that the inequality
	\[\int_X\gamma\wedge g^{n-1}\geq 0\]
	is verified for any Gauduchon metric $g$. Then there exists $f\in L^1(X)$ such that
	\[\gamma+ dd^c f\geq 0\]
	in the sense of currents on $X$. In particular, if $d\gamma= 0$, then the cohomology class defined by it $[\gamma]$ is pseudo-effective.
\end{lemma}

\noindent The proof consists in a clever use of Hahn-Banach theorem. Therefore, in case a line bundle $L$ is not psef, we have  
\begin{equation}\label{eq11}
	\int_Xc_1(L)\wedge g^{n-1}< 0 \end{equation}
for some Gauduchon metric $g$. 
\medskip

We consider the Bott-Chern and Aeppli cohomology 
$$H^{p,q} _{BC} (X) : =\frac{(\ker \partial :  \mathcal{C}^{p,q} (X) \to  \mathcal{C}^{p+1,q} (X) )\cap (\ker \dbar :  \mathcal{C}^{p,q} (X) \to  \mathcal{C}^{p,q+1} (X))}{\Im \ddbar : \mathcal{C}^{p-1,q-1} (X) \to \mathcal{C}^{p,q} (X)} ,$$
and 
$$H^{p,q} _A (X) : =\frac{\ker \ddbar : \mathcal{C}^{p,q} (X) \to \mathcal{C}^{p+1,q+1} (X)}{(\Im \partial :  \mathcal{C}^{p-1,q} (X) \to  \mathcal{C}^{p,q} (X) )+ (\Im \dbar :  \mathcal{C}^{p,q-1} (X) \to  \mathcal{C}^{p,q} (X))} .$$ 
Then for a Gauduchon metric $g$, we have $[g^{n-1}] \in H^{n-1, n-1} _A (X, \mathbb R)$.
We define the psef cone and the movable cone as follows:
\begin{defn} 
	Let $X$ be a $n$-dimensional compact complex manifold. The movable cone $\Mov (X) \subset H^{n-1,n-1} _{A} (X, \mathbb R)$  is  the closed cone generated by $[g^{n-1}]$ for all Gauduchon metric $g$ on $X$. The psef cone $\Psef (X) \subset H^{1,1} _{BC} (X, \mathbb R)$ is the closed cone consists of the $d$-closed $(1,1)$-postive currents on $X$.
	
\end{defn}

Then Lemma \ref{dual} implies that 

\begin{cor}
	Let $X$ be a $n$-dimensional compact complex manifold. Let $\gamma\in H^{1,1} _{BC} (X)$. Then $\gamma\in \Psef (X)$ if and only if $\gamma \cdot \alpha \geq 0$
	for any $\alpha \in \Mov (X)$.
	\end{cor}

\smallskip

Now we recall the degree and slope functions.

\begin{defn} Let $X$ be a compact complex manifold, and let $\cE$ be a coherent, torsion-free sheaf
	of positive rank on $X$. Let $g$ be a Gauduchon metric and let $\alpha :=g^{n-1}$.We define:
	\begin{itemize}
		\item The slope $\mu_\alpha (\cE) :=\frac{1}{\rank \cE} \int_X c_1 (\cE) \wedge g^{n-1}$.
		\smallskip
		
		\item The maximum slope of $\cE$ 
		\[\mu_{\alpha, \rm max}(\cE) := \sup\{\mu_\alpha(\cF) : \cF \subset \cE, \hbox {any nonzero coherent sub-sheaf } \cF \}\]
	
		\item The minimum slope of $\cE$
		\[\mu_{\alpha, \rm min}(\cE) := \inf\{\mu_\alpha(\cQ) : \cE\to \cQ\to 0,\}\]
		where the quotient sheaf $\cQ$ here is coherent, non-zero and torsion-free.
	\end{itemize}
\end{defn}

\smallskip

\noindent We quote next the following result, which implies that many of the main statements in stability theory are still valid in our current context.

\begin{theorem}\cite{Bru1}, \cite{Bru2}\label{max} Let $(X, g)$ be a compact complex manifold endowed with a Gauduchon metric $g$. Then given a coherent, torsion-free sheaf $\cE$ on $X$, there exists a unique maximal $\alpha$-destabilising subsheaf $\cF\subset \cE$.	In particular, $\cF$ is $\alpha$-semistable.
\end{theorem}

\noindent Actually, it is established in the reference above that one can measure the degree of a torsion-free coherent sheaf $\cF$ via an improper integral (over the set by restriction to which the sheaf $\cF$ is a vector bundle). The convergence of the said integral and many other technicalities needed in the proof are fully discussed in \cite{Bru1}, \cite{Bru2}. 
\medskip

\noindent Given two coherent, torsion-free sheaves $\cE_1$ and $\cE_2$, one defines their torsion-free tensor product by the usual formula
\[\cE_1\wh\otimes \cE_2:= \cE_1\otimes\cE_2/\tor,\]
and then the following statement holds true.

\begin{theorem}\cite{CPet}\label{sum} Let $(X, g)$ be a compact complex manifold endowed with a Gauduchon metric $g$. Then we have the equality
	\[\mu_{\alpha, \rm max}(\cE_1\wh\otimes \cE_2)= \mu_{\alpha, \rm max}(\cE_1)+ \mu_{\alpha, \rm max}(\cE_2)\]
	i.e. the maximum slope is additive.  Similar identities hold true if we replace the reflexive tenosr with $(\Lambda^2 \cE)^{\star\star}$ or $\mu_{\alpha, \rm min}$.
\end{theorem}

\noindent This statement relies heavily on the existence of Hermite-Einstein metrics for stable vector bundles, cf. \cite{L-Y} combined with the properties of Harder-Narasimhan filtration.
For the proof we refer to the appendix by by M. Toma in \cite{CPet}. See also \cite{T21} for other results on the topics discussed in this section.
As a consequence of the above two theorems, we have
\begin{cor}\label{foliation}\cite[Lemma 4.10]{CP19}
	Let $(X, g)$ be a compact complex manifold and let $\cF\subset T_X$ be a holomorphic foliation.  Let $g$ be a Gauduchon metric and $\alpha := g^{n-1}$. 
	Let $\cG \subset \cF$ be the unique maximal $\alpha$-destabilising subsheaf of $\cF$.
	If $\mu_\alpha (\cG)>0$, then $\cG$ is a sub-foliation.	 
\end{cor}	

\begin{proof}
	We consider the natural morphism induced by the Lie bracket: $\Lambda^2 \cG \to \cF /\cG$. Since $\cF$ is a foliation, the mrophism is well defined.
	Since $\cG$ is the maximal $\alpha$-destabilising subsheaf, $\cG$ is $\alpha$-semistable.  Together with the assumption $\mu_\alpha (\cG)>0$, we know that $\mu_{\alpha, min} (\Lambda^2 \cG)> \mu_{\alpha, max} (\cF /\cG)$. Therefore the morphism $\Lambda^2 \cG \to \cF /\cG$ is zero. In other words, $\cG$ is a foliation.
\end{proof}

	\begin{proposition}\label{slope}
	Let $X$ be a compact complex manifold, and let $\mathcal{F}$ be a reflexive sheaf on $X$.	
	 If $\cF$ is psef in the sense of Definition \ref{anapsf}, then $\mu_{\alpha, \max} (\mathcal{F}) \geq 0$ for any mobile class $\alpha \in \Mov (X)$. 
	
	Equivalently, if $\mu_{\alpha, \min} (\mathcal{F}) > 0$ for some $\alpha\in \Mov (X)$, then $\cF^*$ is not psef.
\end{proposition}

\begin{proof}
	The proof follows the same idea as in \cite[Section 4]{Ou}. For the reader's convenience, we provide a complete proof here. Thanks to Lemma \ref{dual}, to prove the proposition, it is sufficient to show that $\mu_{\alpha, \max} (\mathcal{F}) \geq 0$, where $\alpha := g^{n-1}$ and $g$ is a Gauduchon metric.
	
	Let $f: \widehat{X} \to X$ be a bimeromorphism such that $E:= f^* \cF / \text{Tor}$ is locally free on $\widehat{X}$. Set $\widehat{X}_0 := f^{-1} (X_0)$. By further blowing up $\widehat{X}$ if necessary, we can assume that there is a bimeromorphism $p: \mathbb{P} (E) \to Y$. Since $f$ is an isomorphism over $\widehat{X}_0$ and $p^*\omega_Y$ is dominated by some hermitian metric $\omega$ on $\mathbb{P} (E)$, 
	$h_\epsilon$ induces a metric on $\mathcal{O}_{\mathbb{P} (E |_{ \widehat{X}_0})} (1)$, still denoted by $h_\epsilon$, such that  
	\[
	i\Theta_{h_\epsilon} (\mathcal{O}_{\mathbb{P} (E |_{ \widehat{X}_0})} (1)) \geq -\epsilon \omega \quad \text{on } \mathbb{P} (E |_{ \widehat{X}_0}) .
	\]
	
	Let $r$ be the rank of $E$, and let $m \in \mathbb{N}^*$. Define $\pi: \mathbb{P} (E) \to \widehat{X}$ as the natural projection. We consider the direct image sheaf
	\[
	\cG_m := \pi_* (K_{\mathbb{P} (E)/ \widehat{X}} \otimes \mathcal{O}_{\mathbb{P} (E)} (m+r+1) \otimes \mathcal{I}(m h_\epsilon))
	\]
	on $\widehat{X}_0$. By using \cite[Lemma 4.4]{CCP}, we know that $\cG_m$ is nontrivial. Moreover, since $\mathcal{O}_{\mathbb{P} (E)} (1)$ is $\pi$-ample, we can find a smooth metric $h_0$ on $\mathcal{O}_{\mathbb{P} (E)} (1)$ such that 
	\[
	i\Theta_{h_0} (\mathcal{O}_{\mathbb{P} (E)} (1)) + \pi^* \omega_{\widehat{X}} > 0,
	\]
	where $\omega_{\widehat{X}}$ is a hermitian metric on $\widehat{X}$. Now we equip $\mathcal{O}_{\mathbb{P} (E)} (m+r+1)$ with the metric $h_\epsilon^m \cdot h_0^{r+1}$. For sufficiently small $\epsilon$, we thus obtain
	\[
	i\Theta_{h_\epsilon^m \cdot h_0^{r+1}} (\mathcal{O}_{\mathbb{P} (E)} (m+r+1)) + (r+2)\pi^* \omega_{\widehat{X}} > 0.
	\]
	Then $h_\epsilon^m \cdot h_0^{r+1}$ induces an $L^2$-metric $h_{\cG_m}$ on $\cG_m$, and by \cite{PT18}, the curvature satisfies
	\begin{equation}\label{curl2}
		i\Theta_{h_{\cG_m}} (\cG_m) \geq - (r+2) \omega_{\widehat{X}} \cdot \Id \quad \text{on } \widehat{X}_0.
	\end{equation}
	
	Since $f$ is an isomorphism over $\widehat{X}_0$, $\cG_m$ is isomorphic to a nonzero subsheaf of $\det \cF \otimes (\otimes^{m+1} \cF)$ on $X_0$. We continue to denote it by $\cG_m$. 
	As $X_0$ is of codimension at least $2$, we can extend $\cG_m$ to be a subsheaf of $ \det \cF \otimes (\otimes^{m+1} \cF)$ on $X$, which we still denote by $\cG_m$. 	Then, from \eqref{curl2}, we obtain
	\[
	\mu_{\alpha} (\cG_m) \geq - (r+2)\int_X \omega_{\widehat{X}}  \wedge g^{n-1}.
	\]
	Since $(r+2)\int_X \omega_{\widehat{X}}  \wedge g^{n-1}$ is independent of $m$, we deduce that
	$
	\mu_{\alpha, \max} (\otimes^{m+1} \cF) \geq C
	$
	for some constant $C$ independent of $m$. 
	Finally, applying Theorem \ref{sum}, we conclude that
	$
	\mu_{\alpha, \max} (\mathcal{F}) \geq \frac{C}{m+1}.
	$
	Letting $m \to +\infty$, we complete the proof.
\end{proof}
\medskip

\noindent We end this section with a few considerations about the notion of \emph{numerical dimension} of a real 
$(1,1)$-class $\{\alpha\}$. 

Consider a Kähler metric $\omega$ on a $n$-dimensional manifold $X$. For each positive $\ep> 0$ we define $\alpha[-\ep\omega]$ to be the set of currents $T\in \{\alpha\}$ in the class $\{\alpha\}$ such that $T$ is of analytic singularity and 
$T\geq -\ep\omega$. The following notion is the natural generalisation of the familiar numerical dimension of a nef line bundle.

\begin{defn}\cite{Bou}
Let $\{\alpha\}$ be a real $(1,1)$-class. The numerical dimension of $\{\alpha\}$ is equal to
\[\nu(X, \{\alpha\}):= \max \big\{k\in \bZ_+ : \lim\!\!\!\!\!\!\sup _{\ep> 0, T\in \alpha[-\ep\omega]}\int_{X\setminus Z_T}T^k\wedge \omega^{n-k}> 0\big\},\]
where $Z_T\subset X$ is the set of singularities of $T$. In case $\{\alpha\}$ does not contain any closed positive current, we simply say that $\nu(X, \{\alpha\})= -\infty$.
\end{defn}

We refer to \cite{Bou} for the basic properties of this notion. The following statement is immediate.
\begin{lemma}\label{ndim}
Let $\{\alpha\}$ and $\{\beta\}$ be two pseudo-effective $(1,1)$-classes. Then we have 
\[\nu(X, \{\alpha+ \beta\})\geq \nu(X, \{\alpha\}).\] 
\end{lemma}
\begin{proof}
Let $T_\ep\in \alpha[-\ep\omega]$ be a family of currents computing the numerical dimension, say $d$, of $\{\alpha\}$. We consider a closed positive current $\Theta\in \{\beta\}$, together with its regularisation $\Theta_\ep\geq -\ep \omega$. The following inequality
\[\int_{X\setminus Z_\ep}(T_\ep+ \Theta_\ep+ 2\ep \omega)^d\wedge \omega^{n-d}\geq \int_{X\setminus Z_\ep}(T_\ep+ \ep \omega)^d\wedge \omega^{n-d}\]
is clear, and by letting $\ep\to 0$ our lemma is proved.
\end{proof}

%%%%%%%%%%%%%%%%%%%%%%%%%%%%%%%%
%%%%%%%%%%%%%%%%%%%%%%%%%%%%%%%%%%%


\begin{thebibliography}{6}
\bibitem[1]{BP} Berndtsson, Bo and P\u{a}un Mihai :\ {\em Bergman kernels and the pseudoeffectivity of relative canonical bundles,} Duke Math. J. 145 (2008), no. 2, 343--378. 
	

	
\bibitem[2]{BM16} Bogomolov, Fedor  and  McQuillan, Michael: {\em Rational curves on foliated varieties.} In Foliation theory in algebraic geometry, Simons Symp., pages 21–51. Springer, Cham, 2016.
	
		\bibitem[3]{Bos01} Bost,  Jean-Benoît: {\em Algebraic leaves of algebraic foliations over number fields.} Publ. Math. Inst. Hautes Études Sci., (93):161–221, 2001.
	
	\bibitem[4]{Bos04} Bost, Jean-Benoît: {\em Germs of analytic varieties in algebraic varieties: canonical metrics and arithmetic algebraization theorems.} In Geometric aspects of Dwork theory. Vol. I, II, pages 371– 418. Walter de Gruyter, Berlin, 2004.
	
\bibitem[5]{Bou} Boucksom, Sébastien	{\em Divisorial Zariski decompositions on compact complex manifolds.}
	Ann. Sci. École Norm. Sup. (4) 37 (2004), no. 1, 45–76.
	
\bibitem[6]{BDPP}  Boucksom, Sébastien;  Demailly, Jean-Pierre; Păun, Mihai  and  Peternell, Thomas {\em The pseudo-effective cone of a compact Kähler manifold and varieties of negative Kodaira dimension} J. Algebraic Geom. 22 (2013), 201-248.

\bibitem[7]{Bru1} Bruasse, Laurent
{\em Harder-Narasimhan filtration on non Kähler manifolds.}
Internat. J. Math. 12 (2001), no. 5, 579–594.

\bibitem[8]{Bru2} Bruasse, Laurent
{\em Filtration de Harder-Narasimhan pour des fibrés complexes ou des faisceaux sans torsion.}
Ann. Inst. Fourier (Grenoble) 53 (2003), no. 2, 541–564.

\bibitem[9]{Brun1}
Brunella, Marco
{\em Feuilletages holomorphes sur les surfaces complexes compactes.}
Ann. Sci. École Norm. Sup. (4) 30 (1997), no. 5, 569–594.

\bibitem[10]{Brun2}
Brunella, Marco
{\em A positivity property for foliations on compact Kähler manifolds.}
Internat. J. Math. 17 (2006), no. 1, 35–43.

	
\bibitem[11]{CP19}   Campana, Frédéric and Păun, Mihai: {\em Foliations with positive slopes and birational stability of orbifold cotangent bundles} Publ. Math. Inst. Hautes Études Sci., 129:1–49, 2019.

\bibitem[12]{C} Campana, Frédéric: {\em Local projectivity of Lagrangian fibrations on Hyperk\"ahler manifolds.} Man. Math. 164 (2021), 589-591.

\bibitem[13]{C21}  Campana, Frédéric: {\em The Bogomolov-Beauville-Yau decomposition theorem for KLT projective varieties with trivial canonical class-without tears.} Bull. Soc. Math. France 1489 (2021).

\bibitem[14]{CPet} Campana, Frédéric and Peternell, Thomas: {\em Geometric stability of the cotangent bundle and the universal cover of projective manifolds.}  Bull.Soc. Math. France 139 (2011), 41-74. With an Appendix by M. Toma. 


\bibitem[15]{CCP} 
Campana, Frédéric; Cao, Junyan and Păun, Mihai
{\em Subharmonicity of direct images and applications.}
Convex and complex: perspectives on positivity in geometry, 53–82.
Contemp. Math., 810, American Mathematical Society, Providence, RI, [2025],


\bibitem[16]{CH} Cao, Junyan and Höring, Andreas: {\em  A decomposition theorem for projective manifolds with nef anticanonical bundle.} Journal für die reine und angewandte Mathematik 2014 (724).

\bibitem[17]{CP17} Cao, Junyan and Păun, Mihai: {\em Kodaira dimension of algebraic fiber spaces over abelian varieties.} Invent. Math., Volume 207, pages 345–387 (2017).

\bibitem[18]{CT}  Collins, Tristan C.  and Tosatti, Valentino:  {\em An extension theorem for Kähler currents with analytic singularities} Annales de la Faculté des sciences de Toulouse : Mathématiques, Série 6, Tome 23 (2014) no. 4, pp. 893-905.

\bibitem[19]{Dem}   Demailly, Jean-Pierre:  {\em Complex analytic and differential geometry} https://www-fourier.ujf-grenoble.fr/~demailly/documents.html

\bibitem[20]{Dem93}  Demailly, Jean-Pierre:  {\em A numerical criterion for very ample line bundles} J. Differential Geom., 37(2):323–374, 1993.

\bibitem[21]{Dem12}   Demailly, Jean-Pierre:  {\em Analytic methods in algebraic geometry} volume 1 of Surveys of Modern Mathematics. International Press, Somerville, MA; Higher Education Press, Beijing, 2012.

\bibitem[22]{DP}  Demailly, Jean-Pierre and  Păun, Mihai: {\em Numerical characterization of the Kähler cone of a compact Kähler manifold} Annals of Mathematics, 159 (2004), 1247–1274

\bibitem[23]{Dru1} Druel, Stéphane: {\em On foliations with nef anti-canonical bundle}
Trans. Amer. Math. Soc. 369, 11 (2017), 7765-7787.

\bibitem[24]{Dru} Druel,  Stéphane: {\em A decomposition theorem for singular spaces with trivial canonical class of dimension at most five} Invent. Math., 211(1):245–296, 2018.

\bibitem[25]{Fav}  Favre, Charles:  {\em Note on pull-back and Lelong number of currents} 
Bulletin de la Société Mathématique de France, Tome 127 (1999) no. 3, pp. 445-458.

\bibitem[26]{Fuj} A. Fujiki: {\em Deformation of uniruled manifolds.} Publ. RIMS 17 (1981), 687-702.

\bibitem[27]{Gau} Gauduchon, P. {\em Le théorème de l’excentricité nulle,} C. R. Acad. Sci. Paris Séri. A-B
285 (1977), no. 5, A387–A390.

\bibitem[28]{GHS} T. Graber, J. Harris and J. Starr, {\em Families of rationally connected varieties, } J. Amer. Math. Soc. 16 (2003), no. 1, 57–67.

\bibitem[29]{HG} Guenancia, Henri: {\em Families of conic Kähler-Einstein metrics.}
Math. Ann. 376 (2020), no. 1-2, 1–37.

\bibitem[30]{HP24} Hacon, C.D.; Paun, M. : {\em On the canonical bundle formula and adjunction for generalised K\"ahler pairs. } arXiv:2404.12007.

\bibitem[31]{HP19} Höring, Andreas  and  Peternell, Thomas {\em Algebraic integrability of foliations with numerically trivial canonical bundle} Invent. Math., 216(2):395–419, 2019.

\bibitem[32]{Hor} Hörmander, Lars: {\em An Introduction to Complex Analysis in Several Variables} North-Holland Mathematical Library, 3rd Revised Edition.

\bibitem[33]{Lam}  Lamari, Ahcène: {\em Courants kählériens et surfaces compactes}
Annales de l'Institut Fourier, Tome 49 (1999) no. 1, pp. 263-285.

\bibitem[34]{Lev} M. Levine: {\em Deformations of uni-ruled varieties.} Duke Math. J. 48 (1981), 467-473.

\bibitem[35]{L-Y} J. Li and S.-T. Yau, {\em Hermitian-Yang-Mills connection on non-Kähler manifolds,} Mathematical aspects of string theory (San Diego, Calif., 1986), Adv. Ser. Math. Phys., vol. 1, World Sci. Publishing, Singapore, 1987, pp. 560–573. 

\bibitem[36]{L-T} M. Lübke and A. Teleman, {\em The Kobayashi-Hitchin correspondence, } World Scientific Publishing Co., Inc., River Edge, NJ, 1995.


\bibitem[37]{Ou}  Ou, Wenhao:  {\em A characterization of uniruled compact Kähler manifolds} arXiv 2501.18088


\bibitem[38]{MP1} Păun, Mihai : {\em On the Albanese map of compact Kähler manifolds
	with numerically effective Ricci curvature} Communications in Analysis and Geometry, Volume 9, Number 1, 35-60, 2001.

\bibitem[39]{MP2} Păun, Mihai: {\em Relative critical exponents, non-vanishing and metrics with minimal singularities.}
Invent. Math. 187 (2012), no. 1, 195–258.

\bibitem[40]{PT18} 
Păun, Mihai  and Takayama, Shigeharu: {\em Positivity of twisted relative pluricanonical bundles and their direct images,} J. Algebraic Geom. 27 (2018), no. 2, 211–272.

\bibitem[41]{Sho} Shokurov, Vyacheslav: {\em A non-vanishing theorem.} Izv. Akad. Nauk SSSR 49 (1985)

\bibitem[42]{Siu} Siu, Yum-Tong.:  {\em Invariance of Plurigenera,} Invent. Math., 134 (1998), 661–673.

\bibitem[43]{T21} Toma, Matei : {\em Properness criteria for families of coherent analytic sheaves}, arXiv, 1710.01484.
		

\end{thebibliography}
\end{document}